\documentclass[11pt, reqno]{amsart}
\usepackage{mathrsfs}
\usepackage{amsthm}
\usepackage{amsfonts}
\usepackage{xr}
\usepackage{graphicx}
\usepackage{amsmath}
\usepackage{epsfig,subfigure}
\usepackage{color}
\usepackage{amssymb}
\usepackage{latexsym}

\newcommand{\SE}{\setcounter{equation}{0} \section}
\newcommand{\be}{\begin{equation}}
\newcommand{\ee}{\end{equation}}
\newcommand{\baa}{\begin{array}}
\newcommand{\eaa}{\end{array}}
\newcommand{\ba}{\begin{eqnarray}}
\newcommand{\ea}{\end{eqnarray}}
\numberwithin{equation}{section}

\newtheorem{thm}{\bf Theorem}[section]
\newtheorem{lem}[thm]{\bf Lemma}
\newtheorem{prop}[thm]{\bf Proposition}

\newtheorem{defi}[thm]{\bf Definition}
\newtheorem{rem}[thm]{\bf Remark}

\newcommand{\me}{\mathrm{e}}

\newcommand{\R}{\mathbb{R}}

\begin{document}
\date{\today}
\title[Stefan problem for the Fisher-KPP equation]{The Stefan problem for the Fisher-KPP equation with unbounded initial range} 
\thanks{Y. Du would like to thank Prof Changfeng Gui for inspiring discussions on this problem. This work was supported by the Australian Research Council, the National Natural Science Foundation of China (11171092, 11571093) and the Innovation Scientists and Technicians Troop Construction Projects of Henan Province (114200510011). }
\author[W. Ding, Y. Du and Z. Guo]{Weiwei Ding$^\dag$, Yihong Du$^\ddag$ and Zongming Guo$^\sharp$}
\thanks{$^\dag$ School of Mathematical Sciences, South China Normal University, Guangzhou 510631, China}
\thanks{$^\ddag$ School of Science and Technology, University of New England, Armidale, NSW 2351, Australia}
\thanks{$^\sharp$ Department of Mathematics, Henan Normal University, Xinxiang, 453007, P.R. China}



\begin{abstract}
We consider the nonlinear Stefan problem
\[
 \left \{ \begin{array} {ll} 
 \displaystyle\medskip u_t-d \Delta u=a u-b
u^2  \;\; & \mbox{for $x \in \Omega (t), \; t>0$},\\
\displaystyle\medskip u=0 \hbox{ and } u_t=\mu|\nabla_x u |^2 \;\;&\mbox{for $x \in \partial\Omega (t), \; t>0$}, \\
u(0,x)=u_0 (x) \;\; & \mbox{for $x \in \Omega_0$},
\end{array} \right.
\]
where  $\Omega(0)=\Omega_0$ is an unbounded smooth domain in $\R^N$, $u_0>0$ in $\Omega_0$ and
$u_0$ vanishes on $\partial\Omega_0$. When $\Omega_0$ is bounded, the long-time behavior of this problem has been rather well-understood by \cite{DG1,DG2,DLZ, DMW}. Here we reveal some interesting different behavior for certain unbounded $\Omega_0$. We also give a unified 
approach for a  weak solution theory to this kind of free boundary problems with bounded or unbounded $\Omega_0$.
\end{abstract}

\maketitle


\SE{Introduction}\label{sec1}
 In this paper, we  investigate the following nonlinear Stefan problem
\begin{equation}\label{eqfrfisher} 
 \left \{ \begin{array} {ll} 
 \displaystyle\medskip u_t-d \Delta u=a u-b
u^2  \;\; & \mbox{for $x \in \Omega (t), \; t>0$},\\
\displaystyle\medskip u=0 \hbox{ and } u_t=\mu|\nabla_x u |^2 \;\;&\mbox{for $x \in \Gamma (t), \; t>0$}, \\
u(0,x)=u_0 (x) \;\; & \mbox{for $x \in \Omega_0$},
\end{array} \right.
\end{equation}
where $\Omega (t) \subset \R^N \; (N \geq 2)$ is a varying domain with
boundary $\Gamma (t)$,  $a,\, b$, $\mu$ and $d$ are
given positive constants, $u_0$ is positive in $\Omega_0$ and vanishes on $\partial \Omega_0$. We are interested in the case that $\Omega (0)=\Omega_0$ is an unbounded smooth domain, which induces extra difficulties from the case of bounded $\Omega_0$, and gives rise to some interesting new phenomena (to be explained below).

 Problem \eqref{eqfrfisher} is an analogue of the classical one-phase Stefan
problem but with a logistic type nonlinear source term on the right
side of the diffusive equation. Such a diffusive equation is
often called a Fisher-KPP equation due to the pioneering works of Fisher \cite{Fisher} and Kolmogorov, Petrovski and Piskunov \cite{KPP}, and has been widely used in the
study of propagation questions. 
In \cite{DG1, DG2, DLZ, DMW}, \eqref{eqfrfisher} with a bounded $\Omega_0$ was used  to describe the spreading of a new or invasive species with population density $u(t,x)$ and population range $\Omega(t)$. The evolution of the spreading front
is determined by the free boundary $\Gamma(t)$, governed by the equation $u_t=\mu|\nabla_x u |^2$ on $\Gamma(t)$, which means the velocity of the
movement of a point $x\in\Gamma (t)$ is given by $\mu |\nabla_x u|\nu_x$, where $\nu_x$ denotes  the unit outward normal vector of $\Omega(t)$ at $x$ (a deduction of this condition from ecological considerations can be found in \cite{BDK}). In one space dimension, such a problem was first considered in \cite{DL}.

Problem \eqref{eqfrfisher} is closely related to the following Cauchy problem:
\begin{equation} \label{introcauchy} 
 \left \{ \begin{array} {ll} 
 \displaystyle\medskip U_t-d \Delta U=a U-b
U^2  \;\; & \mbox{for $x \in \R^N, \; t>0$},\\
U(0,x)=u_0 (x) \;\; & \mbox{for $x \in \R^N$},
\end{array} \right.
\end{equation}
where $u_0(x)$ is given in \eqref{eqfrfisher}  but extended to $\R^N$ with value $0$ outside $\Omega_0$. Indeed, when $\Omega_0$ is bounded, it has been shown in \cite{DG2} that as $\mu\to\infty$, for any fixed $t>0$, $\Omega(t)\to\R^N$ and $u(t,x)\to U(t,x)$, where $u$ and $U$ are the unique solutions of \eqref{eqfrfisher} and \eqref{introcauchy}, respectively. (This holds true for unbounded $\Omega_0$ as well; see Theorem \ref{asymu} below.) Problem \eqref{introcauchy} has been studied extensively as a model for propagation (see, for example, \cite{AW78, Fisher, KPP}).

\smallskip

The basic feature of \eqref{eqfrfisher} with a bounded $\Omega_0$ is given by the following theorem of \cite{DMW}:

{\bf Theorem A.} \begin{enumerate}
\item[(i)] $\Omega(t)$ is expanding:
$\overline\Omega_0\subset \Omega(t)\subset \Omega(s)$ if $0<t<s$;
\item[(ii)] $\Gamma(t)\setminus (\mbox{convex hull of }\overline\Omega_0$) is smooth (e.g., $C^2$);
\item[(iii)] $\Omega_\infty:=\cup_{t>0}\Omega(t)$ is either the entire
space $\R^N$, or it is a bounded set;
\item[(iv)] When $\Omega_\infty$ is bounded,
$\lim_{t\to\infty}\|u(t,\cdot)\|_{L^\infty(\Omega(t))}=0$;
\item[(v)] When
$\Omega_\infty=\R^N$, for all large $t$, $\Gamma(t)$ is a
smooth closed hypersurface in $\R^N$, and there exists a
continuous function $M(t)$ such that
\[
 \Gamma(t)\subset\Big\{x: M(t)-\frac{\pi}{2} d_0\leq |x|\leq
M(t)\Big\},
\]
where $d_0$ is the diameter of  $\Omega_0$.
\end{enumerate}

Moreover, it follows from results of \cite{DG2} that
 there exists $c_*=c_*(\mu)\in (0, 2\sqrt{ad})$
  such that, in case (v) of Theorem A, 
  \[
  \lim_{t\to\infty}M(t)/t=c_*,
  \]
   and
  for any $\epsilon>0$, 
\begin{equation*}
\lim_{t\to\infty} \sup_{|x|\leq (c_*-\epsilon)t} \Big|u(t,x)-\frac ab\Big|=0.
\end{equation*}

The number $c_*$ is usually called the spreading speed, which is the unique positive value of $c$ such that the following problem has a unique solution $q$:
\begin{equation}\label{speedchy}
-dq''+cq'=aq-bq^2, \; q>0 \mbox{ in } (0, \infty),\; q(0)=0,\;  q'(0)=c/\mu.
\end{equation}

 The main purpose of this paper is to reveal some rather different and interesting behavior of \eqref{eqfrfisher} for the case 
 that $\Omega_0$ is unbounded. It turns out that such a case is much more complicated than the $\Omega_0$ bounded case. 
 To keep the paper at a reasonable length, we will only examine some very simple situations of unbounded $\Omega_0$ for the long-time behavior of \eqref{eqfrfisher}.

As in \cite{DG2}, a solution to \eqref{eqfrfisher} will be understood in a certain weak sense (to be made precise in section 2 below). We will show that \eqref{eqfrfisher} has a unique weak solution defined for all $t>0$. As for $\Omega_0$, we will assume that there exist two parallel circular cones 
$\Lambda_1$ and $\Lambda_2$, with the same axis,  such that
\[
\Lambda_1\subset \Omega_0\subset \Lambda_2.
\]
Without loss of generality, we may assume that the common axis of $\Lambda_1$ and $\Lambda_2$ is the $x_N$ axis in $\R^N$. If we denote
\[
e_N:=(0,..., 0, 1)\in \R^N,
\]
then there exist $\phi\in (0, \pi)$ and $\xi_1,\xi_2\in \R^1$ so that
\[
\Lambda_i=\left\{x\in\R^N: \frac{x-\xi_ie_N}{|x-\xi_ie_N|}\cdot e_N>
 \cos\phi\right\},\; i=1,2.
\]
In other words, $\Lambda_i$ has vertex at $\xi_ie_N$ and opening angle $2\phi$. Figure 1 below illustrates a case that $\phi\in (\pi/2, \pi)$ with $\theta=\pi-\phi\in (0, \pi/2)$.
\begin{figure}[h]
\centering
\def\svgwidth{8cm}
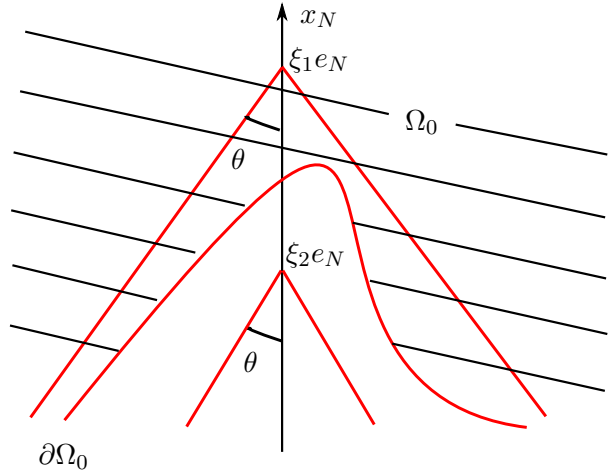
\caption{$\Lambda^\phi+\xi_1e_N\subset \Omega_0\subset \Lambda^\phi+\xi_2e_N$ with $\phi=\pi-\theta\in (\pi/2, \pi)$ }\label{figoutcone}
\end{figure}

For clarity of notations, for $\phi\in (0,\pi)$ and $\xi\in\R^1$, we define
\[\Lambda^\phi:=\left\{x\in\R^N\setminus\{0\}: \frac{x}{|x|}\cdot e_N>\cos\phi\right\} 
\]
and
\[
\Lambda^\phi+\xi e_N:=\Big\{x\in\R^N: x-\xi e_N\in\Lambda^\phi\Big\}.
\]
So $\Lambda^\phi$ and $\Lambda^\phi+\xi e_N$ are parallel cones with vertices at the origin and $\xi e_N$, respectively.

\bigskip

We are now able to describe the main results of this paper.
\begin{thm}\label{thm1}
Suppose that there exist $\phi\in (\pi/2, \pi)$ and $\xi_1>\xi_2$ such that
\begin{equation}\label{Omega0}
\Lambda^{\phi}+\xi_1{e}_N\subset\Omega_0\subset \Lambda^{\phi}+\xi_2{e}_N
\end{equation}
and 
\[
u_0\in C^1(\overline\Omega_0)\cap L^\infty(\Omega_0),\;\liminf_{d(x,\partial\Omega_0)\to\infty} u_0(x)>0.
\]
Then for any given small $\epsilon>0$, there exists $T=T(\epsilon)>0$ such that, for all $t>T$,
\begin{equation}\label{Omega(t)-1}
\Lambda^\phi-\left(\frac{c_*}{\sin\phi}-\epsilon\right)t\,{e}_N\subset \Omega(t)\subset \Lambda^\phi -\left(\frac{c_*}{\sin\phi}+\epsilon\right)t\,{e}_N,
\end{equation}
where $c_*$ is the spreading speed determined by \eqref{speedchy}. 
\end{thm}

The above result indicates that as $t\to\infty$, the free boundary $\Gamma(t)=\partial\Omega(t)$ propagates to infinity in the direction $-e_N$
at roughly the speed $c_*/\sin\phi$, and the shape of the free boundary can be roughly approximated by the boundary of the cone $\Lambda^\phi$.

For any $R>0$, if we denote the $R$-neighbourhood of $\Lambda^\phi$ by $N[\Lambda^\phi, R]$, namely
 \[
 N[\Lambda^\phi, R]:=\big\{x\in\R^N: d(x, \Lambda^\phi)<R\big\},
 \]
 then, it is easily seen that when $\phi\in [\pi/2, \pi)$, we have
 \[
 N[\Lambda^\phi, R]=\Lambda^\phi-\frac{R}{\sin\phi}\,e_N,
 \]
 i.e., $N[\Lambda^\phi, R]$ is a shift of $\Lambda^\phi$ in the direction of $-e_N$ by ${R}/{\sin\phi}$. Thus \eqref{Omega(t)-1} is equivalent to
 \begin{equation}\label{Omega(t)-2}
 N[\Lambda^\phi, (c_*-\tilde \epsilon)t]\subset \Omega(t)\subset N[\Lambda^\phi, (c_*+\tilde\epsilon)t] \mbox{ for } t> T, \;\mbox{ with $\tilde\epsilon:=\epsilon\sin\phi$.}
 \end{equation}

In sharp contrast, if $\phi\in (0, \pi/2)$,  then $N[\Lambda^\phi, R]$ has smooth boundary for all $R>0$, and for large $R>0$, $N[\Lambda^\phi, R]$
becomes very different from $\Lambda^\phi$ geometrically. Indeed, for such $\phi$, it is easily verified that $\partial N[\Lambda^\phi, R]$ can be decomposed into two parts: a spherical part $S_R$ on the sphere $\partial B_R(0)$, and the remaining part $C_R$ on the boundary of the
cone $\Lambda^\phi-\frac R {\sin \phi} e_N$, as follows
\[
\partial N[\Lambda^\phi, R]=S_R\cup C_R,\;  S_R:=\partial B_R(0)\setminus \Lambda^{\phi+\frac\pi 2},\;
C_R:=\Big(\partial\Lambda^\phi-\frac R {\sin \phi} e_N\Big)\cap \Lambda^{\phi+\frac\pi 2}.
\]

The following theorem shows that \eqref{Omega(t)-2} also holds when
$\phi\in (0, \pi/2)$, though the geometric implications of \eqref{Omega(t)-2} are very different from that in Theorem 1.1 now.

\begin{thm}\label{thm2}
In Theorem \ref{thm1} above, if we replace $\phi\in (\pi/2, \pi)$ by $\phi\in (0, \pi/2)$, then for any $\epsilon>0$, there exists $\tilde{T}=\tilde{T}(\epsilon)>0$ such that
\begin{equation}\label{inconeset0}
N\big[\Lambda^\phi, (c_*-\epsilon)t\big]\, \subset  \,{\Omega(t)}\, \subset \,  N\big[\Lambda^\phi, (c_*+\epsilon)t\big] \,\hbox{ for all }  t\geq \tilde{T}.
\end{equation}
\end{thm}

Let us note that \eqref{inconeset0} implies that for all large $t$, $\partial\Omega(t)$ is roughly approximated in shape by
$\partial N[\Lambda^\phi, R]$, which is very different from $\partial\Lambda^\phi$; in particular the propagation of $\partial\Omega(t)$ in the directions $\nu$ satisfying $\nu\cdot e_N\leq \cos (\phi+\frac\pi 2)$ is roughly at speed $c_*$.

\begin{rem} \begin{itemize}
\item[(i)]
If \eqref{Omega0} holds with $\phi=\pi/2$, then much better results than those in Theorems 1.1 and 1.2 can be obtained; it can be shown that the free boundary $\partial\Omega(t)$ converges to a moving hyperplane,
and the solution converges to the corresponding planar semi-wave. These require very different techniques and will be considered elsewhere.
\item[(ii)]
In the case of Theorem 1.1, it is natural to conjecture that the solution $u$ converges to a traveling wave with a V-shaped front {\rm (}free boundary{\rm )}. This and related questions will be investigated in a future work {\rm (\cite{DGWZ})}. 
\end{itemize}
\end{rem}

Let us mention
here that the existence of  ``V-shaped" non-planar traveling wave solutions for the Cauchy problem \eqref{introcauchy} are well known by now (see e.g., \cite{HN,HMR1,HMR2,NT, T}), which lend support to the validity of our conjecture above, but the existence of V-shaped semi-wave with free boundary is technically much more demanding. Note also that the nonlinearities in \cite{HMR1,HMR2,NT, T} are of bistable type, very different from the Fisher-KPP type here.

\bigskip

The rest of the paper is organized as follows. In Section 2, we prove the existence and uniqueness of a weak solution to \eqref{eqfrfisher} by following the approach of \cite{DG2}, where the case of bounded $\Omega_0$  was treated. For unbounded $\Omega_0$, some extra difficulties arise.  In view of possible future applications, here we give a unified approach for  a much more general problem than \eqref{eqfrfisher}, with $\Omega_0$ either bounded or unbounded; see \eqref{eqmain}.

In Section 3, we study the long-time dynamical behavior of problem \eqref{eqfrfisher}. We consider a special type of unbounded $\Omega_0$, namely \eqref{Omega0} holds for some $\phi\in (0,\pi)$. The main results are
Theorems \ref{thm6.1}, \ref{spreadspeed} and \ref{spreadspeed1}. In particular, the conclusions stated in Theorem 1.1 above follow from Theorem \ref{spreadspeed}, and the statements in Theorem \ref{thm2} are consequences of Theorem  \ref{spreadspeed1}. 

The Appendix (Section 4) is concerned with the existence and uniqueness of classical solutions for an auxiliary radially symmetric free boundary problem with initial range the exterior of a ball. These conclusions are used in the proof of Theorem \ref{spreadspeed}.


\section{Weak solutions}

In this section, we extend the
 weak solution theory in \cite{DG2} to cover the case that the initial population range $\Omega_0 $ is 
unbounded. To our knowledge, little is known for problem \eqref{eqfrfisher} with a general unbounded $\Omega_0$. 
For future applications, we consider the following more general problem

\begin{equation}\label{eqmain} 
\left \{ \begin{array} {ll} \displaystyle\medskip  u_t-d \Delta u=g(x,u) \;\; &
\mbox{for $x \in \Omega (t), \; t>0$},\\
\displaystyle\medskip u=0, \;\; u_t=\mu|\nabla_x u|^2 &\mbox{for $x \in \Gamma (t), \; t>0$}, \\
u(0,x)=u_0 (x) \;\; & \mbox{for $x \in \Omega_0$},
\end{array} \right.
\end{equation}
where $\Omega_0$ is a domain in $\R^N$  with Lipschitz boundary $\partial\Omega_0$, the initial function $u_0 (x)$ satisfies
\begin{equation}\label{assumeu0} 
u_0 \in C({\overline {\Omega_0}}) \cap H^1_{loc} (\Omega_0)\cap L^{\infty} (\Omega_0),
\;\; u_0>0 \;\; \mbox{in $\Omega_0$}, \;\;\; u_0=0 \;\;
\mbox{on $\partial \Omega_0$},
\end{equation}
and  the reaction term $g(x,u)$ is assumed to satisfy
\begin{equation}\label{assumeg} 
\left. \begin{array} {ll}
\hbox{(i)} & \smallskip g(x,u) \hbox{ is continuous for } (x,u) \in \R^N \times [0, \infty), \,\,\\
\hbox{(ii)} &\smallskip  g(x,u) \hbox{ is locally Lipschitz in } u \hbox{ uniformly for }x \in \R^N,\,\,\\
\hbox{(iii)} &\smallskip  g(x,0) \equiv 0,\,\,\\
\hbox{(iv)} & \hbox{there exists } K>0 \hbox{ such that }  g(x,u)  \leq  Ku \hbox{ for all } x \in \R^N \hbox{ and } u \geq 0.
\end{array} \right\}
\end{equation}

To describe the weak formulation of  \eqref{eqmain}, it is convenient to start by considering $0<t\leq T$ for some arbitrarily given $T \in (0,
\infty)$. As in \cite{DG2}, the idea is to consider the extended $u$ in the bigger region 
  $H_T:=[0, T] \times \R^N$ by defining $u(t,x)=0$ for $x \in \R^N \backslash \Omega (t)$, $0\leq t\leq T$, and regard it as a weak solution of an associated equation with certain jumping discontinuity. 
 Throughout this section, we denote
\begin{equation*}
\alpha (w)=\left \{\begin{array}{ll}\displaystyle\medskip  w \;\;\;&\mbox{if $w>0$},\\
w-d \mu^{-1} \;\;\; & \mbox{if $w \leq 0$},
\end{array} \right.
\end{equation*}
and
\begin{equation*}
{\tilde u}_0 (x)=\left \{ \begin{array}{ll} \displaystyle\medskip  u_0 (x) \;\;\;
& \mbox{for $x \in \Omega_0$},\\
0 \;\;\; &\mbox{for $x \in \R^N \backslash \Omega_0$}.
\end{array} \right.
\end{equation*}

\begin{defi} \label{defweak} 
Suppose that $u_0$ satisfies \eqref{assumeu0} and $g$
satisfies \eqref{assumeg}. A nonnegative function $u \in
H^1_{loc} (H_T) \cap L^{\infty}(H_T)$ is called a {\it weak} solution of
\eqref{eqmain} over $H_T$ if
\begin{equation}\label{eqweak} 
\displaystyle\medskip \int_0^T \int_{\R ^N} \Big[d \nabla_x u \cdot \nabla_x
\phi-\alpha (u) \phi_t\Big] dx dt-\int_{\R ^N} \alpha ({\tilde
u}_0) \phi (0,x) dx=\int_0^T \int_{\R ^N} g(x,u) \phi dxdt
\end{equation}
for every function $\phi \in C^{\infty}(H_T)$ such that $\phi$ has compact support\footnote{This means that there exists a ball $B_R$ in $\R^N$ such that
$\phi(t,x)=0$ for $t\in [0, T]$ and $x\in\R^N\setminus B_R$.}  and 
$\phi=0$ on $\{T\} \times \R^N$.
\end{defi}

Correspondingly, if $u$ satisfies \eqref{eqweak} with ``=" replaced by ``$\geq$ " (resp. ``$\leq$")  for every test function $\phi \in C^{\infty}(H_T)$ satisfying $\phi\geq 0$ in $H_T$, $\phi=0$ on $\{T\} \times \R^N$ and $\phi$ has compact support, then we call it a {\it weak supersolution} (resp. {\it weak subsolution}) of \eqref{eqmain} over $H_T$. 

Moreover, as in \cite{DG2,Fr1}, for each weak solution (or weak supersolution, or weak subsolution) $u(t,x)$, the function
$\alpha (u(t,x))$ is defined as $u(t,x)$ if $u(t,x)>0$; at points
where $u(t,x)=0$ the function $\alpha (u(t,x))$ is only required to
satisfy $-d \mu^{-1} \leq \alpha (u(t,x)) \leq 0$ and to be such
that it is altogether a measurable function over $H_T$. However, if $w(x)$ is
continuous and positive in $\Omega_0$ and identically zero in $\R^N
\backslash \Omega_0$,  then we understand that $\alpha (w)=-d \mu^{-1}$ on $\R^N
\backslash \Omega_0$.

\begin{rem}
Definition \ref{defweak} is an extension of \cite[Definition 2.1]{DG2} for weak solution of problem \eqref{eqmain} with bounded $\Omega_0$ to the 
case that $\Omega_0$ is generally unbounded. Indeed, the choice of the test function in \eqref{eqweak} implies that if $\Omega_0$ 
is bounded and if $u$ is a weak solution to \eqref{eqmain} over $H_T$  by Definition \ref{defweak}, then the restriction of $u$ 
to $G\times [0,T]$ is a weak solution of \eqref{eqmain} over $G\times [0,T]$ by \cite[Definition 2.1]{DG2}, where $G$ is any given 
bounded domain in  $\R^N$ such that $\Omega(t)$ stays inside $G$ for $0\leq t\leq T$. 
\end{rem}

By a {\it classical solution} of problem \eqref{eqmain} for $0<t\leq T<\infty$, we mean a pair  $(u(t,x), \Omega(t))$ such that 
$\bigcup_{0\leq t\leq  T}\partial\Omega(t)$ is a $C^1$ hypersurface in $\R^{N+1}$,  $u$, $\nabla_x u$ are continuous in $\bigcup_{0\leq t\leq T}\overline{\Omega(t)}$, and $u_t$, $u_{x_ix_j}$ are continuous in $\bigcup_{0<t\leq T}\Omega(t)$, and all the equations in \eqref{eqmain} are satisfied in the classical sense.
In this case, there exists $\Phi\in C^1\left(\overline{\bigcup_{0<t\leq T}\Omega(t)}\right)$ such that
$$\Phi(t,x)=0\; \hbox{ on } \Gamma (t),\;\;\;  \nabla_x \Phi (x,t)\neq 0 \;\; \mbox{on $\Gamma (t)$}, \;\;\; \Phi
(t,x)<0 \;\;\; \mbox{in $\Omega (t)$},$$
and it follows from
\[
u=0 \mbox{ and } u_t=\mu|\nabla_x u|^2 \mbox{ on } \Gamma(t)
\]
that
\[
\Phi_t=\mu \nabla_x u\cdot \nabla_x \Phi \mbox{ for } x\in \Gamma(t).
\]

\begin{thm} \label{weakclassic}
 {\rm (i)} Assume that $(u(t,x), \Omega(t))$ is a classical solution of \eqref{eqmain} for $0<t\leq T$.
 Then
 $$ w(t,x):=\left \{ \begin{array}{ll}  \displaystyle\medskip u(t,x) \;\;\; &\mbox{for $x
 \in \Omega (t)$, $0<t \leq T$}, \\
  \displaystyle 0 \;\;\; &\mbox{for $x \in \R^N \backslash \Omega (t)$, $0<t\leq T$}
 \end{array} \right.$$
 is a weak solution of \eqref{eqmain} over $H_T$.

 {\rm (ii)} Let $w$ be a weak solution of \eqref{eqmain} on $H_T$. Assume that
 there exists a $C^1$ function $\Phi$ in ${H_T}$
 satisfying
 $$\Omega (t): = \Big\{x \in \R^N: \; w(t,x)>0\Big\}=\Big\{x \in \R^N: \; \Phi
 (t,x)<0\Big\}$$
with $\Omega (0)=\Omega_0$, and
$$\nabla_x \Phi \neq 0 \;\;\; \mbox{on\; $\Gamma (t) := \partial
\Omega (t)$}.$$
Setting $u=w$ in $\bigcup_{0<t\leq T} {\overline {\Omega (t)}}$, and
assume that $u$, $\nabla_x u$ are continuous in $\bigcup_{0 \leq
t\leq T} {\overline {\Omega (t)}}$ and that $\nabla_x^2 u$, $u_t$ are
continuous in $\bigcup_{0<t\leq T} \Omega (t)$. Then $(u(t,x),
\Omega(t))$ is a classical solution of \eqref{eqmain} for $t\in (0, T]$.
\end{thm}

\begin{proof} The proof follows  that of  \cite[Theorem 2.3]{DG2} with some obvious modifications, and we omit the details. 
\end{proof}

Next we prove the existence and uniqueness of weak solution to problem \eqref{eqmain}.  The strategies of the proof are adapted from the approximation method of \cite{Ka} as used in \cite{DG2, Fr1}, but new techniques are required to handle difficulties arising from the unboundedness of $\Omega_0$.
We begin with the uniqueness part, since the existence proof will use some of the arguments in the uniqueness proof.

\begin{thm}\label{unique} 
Suppose that $\mu_1 \geq \mu_2>0$, $u_1$ is a weak supersolution of \eqref{eqmain} over $H_T$ with $\mu=\mu_1$ and  $u_2$ is a weak subsolution with $\mu=\mu_2$, 
where the initial data $u_0$ and $\Omega_0$ are shared. Then $u_1 \geq u_2$ a.e. in $H_T$. In particular, problem \eqref{eqmain} can have at most one weak solution over $H_T$.
\end{thm}

\begin{proof}
We complete the proof in three steps.

\smallskip

\noindent
{\bf Step 1:}  {\it Setup of the approximation method.}
 
  With $u_1$ and $u_2$ as given in the statement of the theorem, we define
  $$ 
  \ell (t,x)=\left \{ \begin{array}{ll} \displaystyle\medskip\frac{g(x,u_2 (t,x))-g(x,
u_1 (t,x))}{u_2 (t,x)-u_1 (t,x)} \;\;& \mbox{if $u_2 (t,x) \neq u_1
(t,x)$}, \\
0 \;\;& \mbox{if $u_2 (t,x)=u_1 (t,x)$},
\end{array} \right.
$$
and for $i=1,2$, let $\alpha_i(u)$ denote $\alpha(u)$ with $\mu=\mu_i$, and
$$
e (t,x)=\left \{ \begin{array}{ll} \displaystyle\medskip\frac{u_2 (t,x)-u_1
(t,x)}{\alpha_2 (u_2 (t,x))-\alpha_1 (u_1 (t,x))} \;\;\; &\mbox{if
$u_2 (t,x) \neq u_1 (t,x)$},\\
0 \;\;\; &\mbox{if $u_2 (t,x)=u_1 (t,x)$}.
\end{array} \right.
$$
  We then have 
\begin{equation}\label{difference} 
 \int_0^T \int_{\R^N} \big[\alpha_2(u_2)- \alpha_1(u_1)\big]\big(\phi_t+de\Delta \phi +e\ell\phi \big) dxdt \geq d(\mu_2^{-1}- \mu_1^{-1})\int_{\R^N \backslash {\overline {\Omega_0}}} \phi(0,x) dx 
\end{equation}
for every nonnegative $\phi \in C^{\infty}(H_T)$ such that $\phi$ has compact support,  and
$\phi=0$ on $\{T\} \times \R^N$.  

It is easily checked that if we write 
$$\alpha_2 (u_2 (t,x))-\alpha_1 (u_1 (t,x))=\bar{\alpha}(t,x)\big[u_2(t,x)-u_1(t,x)\big]$$ when $u_2(t,x)\neq u_1(t,x)$, 
then $\bar{\alpha}(t,x) \geq 1 $ for almost all $(t,x)\in H_T$. And hence, 
\begin{equation}\label{bde}
0 \leq e(t,x) \leq  1 \;\; \mbox{for almost all $(t,x)
\in H_T$}.
\end{equation}
Moreover, by the condition \eqref{assumeg} on $g$ and the fact that $u_1, u_2 \in
L^\infty (H_T)$, it follows that $\ell \in L^{\infty}(H_T)$. We then approximate $e$ and $\ell$ by smooth functions $e_m\in C^{\infty}(H_T)$ and $\ell_m\in C^{\infty}(H_T)$, respectively, such that for any $R>0$, 
\begin{equation}\label{appfns}
\| e - e_m\|_{L^2([0,T]\times B_R)}\to 0,\quad  \| \ell - \ell_m\|_{L^2([0,T]\times B_R)}\to 0\,\,\hbox{ as } m\to\infty,
\end{equation}
and that 
\begin{equation}\label{appcond}
\inf_{H_T} e_m\geq \frac{1}{m}, \quad \|e_m \|_{L^{\infty}(H_T)}\leq C_1,\quad   \|\ell_m \|_{L^{\infty}(H_T)}\leq C_1,
\end{equation}
for some positive constant $C_1$ independent of $m$. Here and in what follows, we use $B_R$ to denote the ball with center the origin and radius $R$.

Now, let $R_0>0$ be fixed, and choose a nonnegative function $f \in C^{\infty}(H_T)$ with $f(t,x)=0$ for all $|x|\geq R_0$ and $0\leq t \leq T$. For any $m\geq 1$ and $R> R_0+1$, we consider the following backward parabolic equation 
\begin{equation} \label{backward} \left \{ \begin{array}{ll} 
\displaystyle\medskip\frac{\partial \phi}{\partial t}+d e_m \Delta \phi+e_m \ell_m \phi=f \;\;\;&
\mbox{in $[0,T)\times B_R$},\\
\displaystyle\medskip \phi=0 \;\;\; &\mbox{on $\{T\} \times B_R$},\\
\phi=0 \;\; &\mbox{on $[0,T] \times \partial B_R$}.
\end{array} \right.
\end{equation}
This is a nondegenerate problem and it has a unique smooth solution $\phi_m$ (see e.g., \cite{LSU}). Furthermore, the parabolic maximum principle (applied to $\phi_m(T-t,x)$) implies that 
$$\phi_m\leq 0 \hbox{ in } [0,T]\times B_R.$$ 

Next, for $0<\epsilon\ll 1$, we take a cutoff function $\xi_{\epsilon}\in C_0^{\infty}(\R^N)$ such that
\begin{equation}\label{modifier1}
0\leq \xi_{\epsilon} \leq 1 \hbox{ in } \R^N, \quad  \xi_{\epsilon} =1 \hbox{ in } B_{R-2\epsilon}, \quad \xi_{\epsilon} =0 \hbox{ in } \R^N \backslash B_{R-\epsilon}, \
\end{equation}
 and that
 \begin{equation}\label{modifier2}
 \| \nabla \xi_{\epsilon} \|_{L^{\infty}(\R^N)} \leq C_2\epsilon^{-1},\quad    \|\Delta \xi_{\epsilon} \|_{L^{\infty}(\R^N)}  \leq C_2\epsilon^{-2} 
 \end{equation}
 for some positive constant $C_2$ independent of $\epsilon$. Then for any $m>1$ and $0<\epsilon \ll 1$, define
 $$ \psi_m^{\epsilon} (t,x)=\left \{\begin{array}{ll} \displaystyle\medskip\phi_m(t,x)\xi_{\epsilon}(x) \;\;\;&\mbox{in $[0,T]\times B_R $},\\
0  \;\;\; & \mbox{in  $[0,T]\times (\R^N  \backslash B_{R})$}.
\end{array} \right.$$
Clearly, $\psi_m^{\epsilon}\leq 0$, it belongs to $ C^{\infty}(H_T)$,  and  vanishes in $(\{T\} \times \R^N)\cup ([0,T]\times (\R^N  \backslash B_{R}))$. 

We may now take $-\psi_m^{\epsilon}$ as a text function in \eqref{difference} to obtain, due to  $\mu_1\geq \mu_2>0$, 
\begin{equation*}
 \int_0^T \int_{\R^N} \big[\alpha_2(u_2)- \alpha_1(u_1)\big]\Big((\psi_m^{\epsilon})_t+de\Delta \psi_m^{\epsilon} +e\ell\psi_m^{\epsilon} \Big) dxdt \leq 0. 
\end{equation*}
Since $\xi_{\epsilon}f\equiv f$ in $H_T$, it then follows that 
\begin{equation}\label{mainesti}
\left.\begin{array}{ll}
& \displaystyle \medskip  \int_0^T \int_{\R^N}\big[\alpha_2(u_2)- \alpha_1(u_1)\big] f dxdt \\
=&\displaystyle \medskip   \int_0^T \int_{B_R}\big[\alpha_2(u_2)- \alpha_1(u_1)\big]\Big( \frac{ \partial\phi_m}{\partial t}+d e_m \Delta \phi_m+e_m \ell_m \phi_m\Big) \xi_\epsilon dx dt\\
\leq & I_m + J_m + K_m,
\end{array}\right. 
\end{equation}
where 
$$I_m=I_m(\epsilon):= \int_0^T \int_{B_R}\big[\alpha_2(u_2)- \alpha_1(u_1)\big] d\xi_{\epsilon}(e_m-e)\Delta\phi_m dxdt,  $$
$$J_m=J_m(\epsilon):= \int_0^T \int_{B_R}\big[\alpha_2(u_2)- \alpha_1(u_1)\big] \xi_{\epsilon}(e_m\ell_m-e\ell)\phi_m dxdt,$$
and 
$$K_m=K_m(\epsilon):= \int_0^T \int_{B_R}\big[\alpha_2(u_2)- \alpha_1(u_1)\big] \Big(-de\phi_m\Delta \xi_{\epsilon}-2de\nabla \phi_m\cdot\nabla \psi_{\epsilon}\Big)dxdt.$$

Our aim is to show, through suitable estimates on $I_m,\; J_m$ and $K_m$,  that
\[
 \int_0^T \int_{\R^N}\big[\alpha_2(u_2)- \alpha_1(u_1)\big] f dxdt\leq 0,
 \]
and the conclusions of the theorem would then follow easily from this inequality.

\smallskip
\noindent
{\bf Step 2:} {\it Estimates of $I_m$, $J_m$ and $K_m$.}

 We first consider $I_m$. Let 
$$C_3:=\|u_1\|_{L^{\infty}(H_T)}+\|u_2\|_{L^{\infty}(H_T)}+d\mu_1^{-1}+d\mu_2^{-1}.$$ 
By the H\"{o}lder inequality, we have 
\begin{equation*} 
\left.\begin{array}{ll}
I_m\!\! &\displaystyle\medskip  \leq dC_3 \int_0^T\int_{B_R} |e_m-e||\Delta\phi_m|dxdt\\
&\displaystyle\medskip \leq dC_3 \Big(\int_{0}^T\int_{B_R} \frac{|e_m-e|^2}{|e_m|}dxdt\Big)^{\frac{1}{2}}\Big(\int_{0}^T\int_{B_R} |e_m||\Delta\phi_m|^2 dxdt\Big)^{\frac{1}{2}}\\
&\displaystyle\medskip \leq dC_3 \big\| e_m-e\big\|_{L^2([0,T]\times B_R)}^{\frac{1}{2}} \Big(\int_{0}^T\int_{B_R} \frac{|e_m-e|^2}{|e_m|^2}dxdt\Big)^{\frac{1}{4}}\Big(\int_{0}^T\int_{B_R} |e_m||\Delta\phi_m|^2 dxdt\Big)^{\frac{1}{2}}.
\end{array}\right.
\end{equation*}   
It follows from the proof of  \cite[Lemma 3.7]{DG2} that  there is a positive constant $C_4=C_4(T,f,R)$ such that 
$$\Big(\int_{0}^T\int_{B_R} |e_m||\Delta\phi_m|^2 dxdt\Big)^{\frac{1}{2}} \leq C_4.$$
Moreover, by the same analysis as that used in the proof of \cite[Lemma 5]{CH}, we may require that the approximation sequence $e_m$ satisfies 
additionally
$$
\Big\| \frac{e}{e_m}\Big\|_{L^2([0,T]\times B_R)} \leq C_5
$$
for some positive constant $C_5=C_5(T,R).$ We thus obtain
\begin{equation}\label{estiim}
I_m\leq  C_6 \big\| e_m-e\big\|_{L^2([0,T]\times B_R)}^{\frac{1}{2}},
\end{equation}
with $C_6=d(T|B_R| +C_5^2)^{1/4}C_3C_4$ independent of $m$.

Next, it is easily seen that 
\begin{equation*} 
\left.\begin{array}{ll}
\medskip  J_m\;\leq \;C_3\|\phi_m \|_{L^{\infty}([0,T]\times B_{R})}
&\!\!\!\!\Big(\| e_m\|_{L^2([0,T]\times B_R)} \big\| \ell_m-\ell\big\|_{L^2([0,T]\times B_R)}\\
&\;\;\;\;\;\; + \| \ell \|_{L^2([0,T]\times B_R)} \big\| e_m-e\big\|_{L^2([0,T]\times B_R)}  \Big). 
\end{array}\right.
\end{equation*}   
By the comparison arguments used in \cite[Lemma 3.6]{DG2}, we have 
$$\|\phi_m \|_{L^{\infty}([0,T]\times B_{R})} \leq C_7,$$
for some positive constant $C_7=C_7(T,f)$. Hence, by setting $C_8=(T|B_R|)^{1/2}C_1C_3C_7$, we obtain 
\begin{equation}\label{estijm}
J_m\leq C_8\Big( \big\| \ell_m-\ell\big\|_{L^2([0,T]\times B_R)} + \big\| e_m-e\big\|_{L^2([0,T]\times B_R)}  \Big).
\end{equation}

It remains to estimate $K_m$. Making use of \eqref{bde}, \eqref{modifier1} and \eqref{modifier2}, we have 
\begin{equation}\label{estikmfr} 
\left.\begin{array}{ll}
K_m &\displaystyle\medskip  \leq dC_3 \int_0^T\int_{B_R}\Big( \big|\phi_m\Delta \xi_{\epsilon}\big| + 2\big| \nabla \phi_m\cdot\nabla \xi_{\epsilon}\big|\Big)dxdt \\
&\displaystyle\medskip \leq dC_2C_3 \Big(\int_{0}^T\int_{B_{R-\epsilon}\backslash B_{R-2\epsilon}} \Big( \frac{|\phi_m|}{\epsilon^2}+ \frac{2|\nabla \phi_m|}{\epsilon}\Big)dxdt.
\end{array}\right.
\end{equation} 
Since $\phi_m(t,x)=0$ on $[0,T]\times \partial B_R$, it follows by letting $\epsilon \to 0$ in \eqref{estikmfr} that 
\begin{equation}\label{estikmau}
\limsup_{\epsilon\to 0} K_m(\epsilon) \leq C_9 R^{N-1} \sup_{0\leq t\leq T,\,x\in\partial B_R} \Big|\frac{\partial \phi_m(t,x)}{\partial \nu_x} \Big|,
\end{equation}  
where $C_9=C_9(T,C_2,C_3)$ is a positive constant independent of $m$ and $R$, and 
$\nu_x$ is the outward unit normal vector of $B_R$ at $x\in \partial B_R$. 

We next estimate  $\partial \phi_m(t,x)/\partial \nu_x$ at the boundary $\partial B_R$ by making use of  barrier 
functions inspired by the proof of \cite[Theorem 2.1]{DV}. Let 
$$v(t,x):=-C_{10}\me^{-\beta_1t}(1+|x|^2)^{-\beta_2} \,\hbox{ for } (t,x)\in [0,T]\times B_R, $$
where $C_{10}$, $\beta_1$ and $\beta_2$ are positive constants  independent of $R$, to be chosen later. 
It is straightforward to verify that
$$v_t=-\beta_1 v\quad \hbox{and} \quad\Delta v \geq   4(\beta_2^2+\beta_2)v \,\,\hbox{ for } (t,x)\in [0,T]\times B_R, $$
and thus,
\begin{equation*} 
 v_t+d e_m \Delta v+e_m \ell_m v\geq  -v \Big (\beta_1-4(\beta_2^2+\beta_2)de_m-e_m\ell_m \Big) \,\hbox{ for } (t,x)\in [0,T]\times B_R.
 \end{equation*} 
Due to \eqref{appcond} and the fact that $f(t,x)=0$ for all $|x|\geq R_0$ and $0\leq t \leq T$, we may choose 
$$\beta_1= 4(\beta_2^2+\beta_2)dC_1+C_1^2+1,$$
and 
$$ C_{10}=\me^{\beta_1 T}\big(1+R_0^2\big)^{\beta_2} \max_{(t,x)\in [0,T]\times B_{R_0} } |f(t,x)|,$$
 and thus conclude that
\begin{equation*} \left \{ \begin{array}{ll} 
\medskip v_t+d e_m \Delta v+e_m \ell_m v \geq f  \;\;\;&
\mbox{in $[0,T)\times B_R$},\\
\displaystyle\medskip v<0 \;\;\; &\mbox{on $\{T\} \times B_R$},\\
v < 0 \;\; &\mbox{on $[0,T] \times \partial B_R$}.
\end{array} \right.
\end{equation*}
It then follows from the parabolic maximum principle applied to problem \eqref{backward} that 
\begin{equation}\label{phimup}
 v(t,x) \leq  \phi_m(t,x)\leq 0  \,\hbox{ for } (t,x)\in [0,T)\times B_R.
\end{equation}

We next consider the function 
$$v^*(t,x):=\frac{-C_{10}\me ^{-\beta_1 t} \sigma(|x|)}{\sigma(R-1)\big(1+(R-1)^2\big)^{\beta_2}}\,\, \hbox{ for } (t,x)\in [0,T]\times B_R\backslash B_{R-1}, $$
where 
 $$ \sigma(\rho)=\left \{\begin{array}{ll} \displaystyle\medskip
\rho^{2-N}-R^{2-N}\;\;\;&\mbox{ if } N\geq 3,\\
 \log {\big(R/\rho\big)}  \;\;\; & \mbox{ if }  N=2.
\end{array} \right.$$
The function $v^*(t,x)$ satisfies 
$\Delta v^*=0$ and $$v^*_t+e_m \ell_m v^*=-v^*(\beta_1-e_m\ell_m)\geq 0 \mbox{
in $[0,T)\times B_R\backslash B_{R-1}$.}
$$
  Since $f(t,x)=0$ for all $0\leq t\leq T$ and $x\in B_R\backslash B_{R-1}$ (note that $R>R_0+1$), it then follows that 
$$ v^*_t+d e_m \Delta v^*+e_m \ell_m v^* \geq f  \, \hbox{ in $[0,T)\times B_R\backslash B_{R-1}$}.$$
On the other hand, we have, in view of \eqref{phimup},
\begin{equation*} \left \{ \begin{array}{ll} 
\displaystyle\medskip v^*<0 \;\;\; &\mbox{on $\{T\} \times B_R\backslash B_{R-1}$},\\
\displaystyle\medskip v^* = 0 \;\; &\mbox{on $[0,T] \times \partial B_R$},\\
v^* =-C_{10}\me^{-\beta_1 t}\big(1+(R-1)^2\big)^{-\beta_2}=v \leq \phi_m \;\; &\mbox{on $[0,T] \times \partial B_{R-1}$}.
\end{array} \right.
\end{equation*}  
It then follows from the parabolic maximum principle again that
$$ v^*(t,x) \leq \phi_m(t,x)\leq 0 \,\,\hbox{ for }  (t,x)\in [0,T]\times B_R\backslash B_{R-1}.$$
Since both $\phi_m$ and $v^*$ vanish on $\partial B_R$, we thus obtain
$$\frac{\partial v^*(t,x)}{\partial \nu_x} \geq  \frac{\partial \phi_m(t,x)}{\partial \nu_x}\geq 0\,\, \hbox{ for }  (t,x)\in [0,T] \times \partial B_{R}.$$
This together with the estimate \eqref{estikmau} implies that
\begin{equation*}
\left. \begin{array}{ll} 
\displaystyle\medskip
\limsup_{\epsilon\to 0} K_m(\epsilon) &\displaystyle\medskip \leq \,\,C_9 R^{N-1} \sup_{0\leq t\leq T,\,x\in\partial B_R} \Big|\frac{\partial v^*(t,x)}{\partial \nu_x} \Big|\\
&\,\,\displaystyle\medskip = \frac{C_9C_{10}R^{N-1}|\sigma'(R)| }{\sigma(R-1)\big(1+(R-1)^2\big)^{\beta_2}}.
\end{array} \right.
\end{equation*} 
Thus, if we take $\beta_2=N/2$, then there exists some positive constant $C_{11}=C_{11}(N, C_9, C_{10})$ independent of $R$ and $m$ such that 
\begin{equation}\label{estikm}
\limsup_{\epsilon\to 0} K_m (\epsilon) \leq  C_{11} R^{-1}.
\end{equation} 

\smallskip \noindent
{\bf Step 3:} {\it Completion of the proof.}

 Combining the estimates in \eqref{estiim}, \eqref{estijm} and \eqref{estikm}, we obtain from \eqref{mainesti} that  
\begin{equation*}
\left.\begin{array}{ll}
& \displaystyle \medskip  \int_0^T \int_{\R^N}\big[\alpha_2(u_2)- \alpha_1(u_1)\big] f dxdt \\
\leq &\displaystyle \medskip   C_6 \big\| e_m-e\big\|_{L^2([0,T]\times B_R)}^{\frac{1}{2}} + C_8\Big( \big\| \ell_m-\ell\big\|_{L^2([0,T]\times B_R)} + \big\| e_m-e\big\|_{L^2([0,T]\times B_R)}  \Big) + C_{11}R^{-1}.
\end{array}\right. 
\end{equation*}
In view of  \eqref{appfns} and the facts that $C_6,\,C_8$ depend on $R$ but not on $m$ and that $C_{11}$ is independent of $m$ and $R$, passing to the limit as $m\to\infty$ followed by letting $R\to\infty$ in the above inequality, we obtain 
\begin{equation*}
 \int_0^T \int_{\R^N}\big[\alpha_2(u_2)- \alpha_1(u_1)\big] f dxdt \leq 0.
\end{equation*}
Since it holds for all nonnegative smooth function $f(t,x)$ with compact support in $x$, it follows that $\alpha_2(u_2)\leq \alpha_1(u_1)$, and hence in the a.e. sense  $u_1(t,x)>0$ whenever $u_2(t,x)>0$, and $u_1\geq u_2$ a.e. in $H_T$. The proof of Theorem \ref{unique} is now complete.
\end{proof}

We now consider the existence of weak solutions of problem \eqref{eqmain} over $H_T$. 

\begin{thm}\label{existence}
There exists a unique weak solution $w$ of \eqref{eqmain} over $H_T$.
\end{thm}

The proof of this theorem also follows the approximation arguments used in \cite{DG2,Fr1} and we only provide the details where considerable changes are required. Before giving the proof, we first introduce some notations and approximation functions. 
 Let $\{\alpha_m (w)\}_{m\in\mathbb{N}}$ be a sequence of smooth functions such that 
 \begin{equation*}
 \alpha_m (w) \to \alpha (w) \hbox{ uniformly in any compact subset of } \R^1 \backslash \{0\},
 \end{equation*}
 and 
 \begin{equation*}
 \alpha_m (0) \to -d \mu^{-1}, \quad w-d \mu^{-1} \leq \alpha_m (w) \leq w \hbox{ for all } w \in \R^1.
 \end{equation*}
We may choose $\alpha_m (w)$ in such a way that
\begin{equation}\label{restriction} 
\alpha_m'(w) \geq 1.
\end{equation}

We now consider the following sequence of approximating problems:
\begin{equation}\label{exappx} 
\left \{ \begin{array}{ll} 
 \displaystyle\medskip\frac{\partial \alpha_m
(w)}{\partial t}-d \Delta w=g(x, w) \;\;\;& \mbox{in $(0,T]\times \R^N$},\\
 \displaystyle w (0,x)={\tilde u}_0 (x) \;\;\; &\mbox{in $\R^N$}.
\end{array} \right.
\end{equation}

  For any fixed $m\in\mathbb{N}$, we call a function ${w}$ a {\it bounded supersolution} to problem \eqref{exappx}  if 
${w}\in C^{1,2}((0,T]\times\R^N)\cap C(H_T)\cap L^{\infty}(H_T)$  and it satisfies 
\begin{equation}\label{exappsup}
\left \{ \begin{array}{ll} 
 \displaystyle\medskip\frac{\partial \alpha_m
({w})}{\partial t}-d \Delta {w}\geq g(x, {w}) \;\;\;& \mbox{in $(0,T]\times \R^N$},\\
 \displaystyle {w}(0,x)\geq {\tilde u}_0 (x) \;\;\; &\mbox{in $\R^N$}.
\end{array} \right.
\end{equation}
Such a $w$ is called a {\it bounded subsolution} to problem \eqref{exappx} if the reversed inequalities in \eqref{exappsup} are satisfied. We have the following comparison result.

\begin{lem}\label{exappcom} 
Let $m\in \mathbb{N}$ be fixed. 
Assume that $\overline{w}$ and $\underline{w}$ are bounded supersolution and subsolution of problem \eqref{exappx}, respectively. Then $\overline{w}\geq \underline{w}$ in $H_T$.
\end{lem} 

\begin{proof}
It is easily verified from \eqref{exappx} that $\overline{w}$ satisfies 
\begin{equation}\label{weakapp} 
\left.\begin{array}{ll} 
 \displaystyle\medskip \int_0^T \int_{\R ^N} \Big[d \nabla_x \overline{w} \cdot \nabla_x\phi-\alpha_m (\overline{w}) \phi_t\Big] dx dt
\!\!& \displaystyle\medskip-\,\int_{\R ^N}\alpha_m(\overline{w}(0,x))\phi (0,x) dx\\
 \!\!& \displaystyle\medskip \geq\, \int_0^T \int_{\R ^N} g(x,\overline{w}) \phi dx dt
\end{array}\right.
\end{equation}
for every nonnegative function $\phi \in C^{\infty}(H_T)$ such that $\phi$ has compact support and  
$\phi=0$ on $\{T\} \times \R^N$. Analogously, $\underline{w}$ satisfies \eqref{weakapp} with the inequality sign reversed. Subtracting these two inequalities, due to \eqref{restriction} and the fact that $\overline{w}(0,x)\geq \underline{w}(0,x)$ in $\R^N$, we obtain 
\begin{equation*}
 \int_0^T \int_{\R^N} \big[\alpha_m(\underline{w})- \alpha_m(\overline{w})\big]\big(\phi_t+d\tilde{e}\Delta \phi +\tilde{e}\tilde{\ell}\phi \big) dxdt \geq 0, 
\end{equation*}  
where 
$$ \tilde{\ell} (t,x)=\left \{ \begin{array}{ll} \displaystyle\medskip\frac{g(x,\underline{w} (t,x))-g(x,
\overline{w} (t,x))}{\underline{w} (t,x)-\overline{w} 
(t,x) (t,x)} \;\;& \mbox{if $\underline{w}(t,x) \neq \overline{w} 
(t,x)$}, \\
0 \;\;& \mbox{if $\underline{w} (t,x)=\overline{w} (t,x)$},
\end{array} \right.$$
and
$$\tilde{e}(t,x)=\left \{ \begin{array}{ll} \displaystyle\medskip\frac{\underline{w} (t,x)-\overline{w} 
(t,x)}{\alpha_m (\underline{w} (t,x))-\alpha_m (\overline{w}  (t,x))} \;\;\; &\mbox{if
$\underline{w} (t,x) \neq \overline{w}  (t,x)$},\\
0 \;\;\; &\mbox{if $\underline{w} (t,x)=\overline{w} (t,x)$}.
\end{array} \right.$$
Since $\overline{w},\,\underline{w} \in L^{\infty}(H_T)$ and since $\alpha_m'(w) \geq 1$ by \eqref{restriction}, it then follows from the same arguments as those used in the proof of Theorem \ref{unique} that 
$$\alpha_m(\underline{w}(t,x))- \alpha_m(\overline{w}(t,x)) \leq 0 \hbox{ for almost all } (t,x)\in H_T.$$ This immediately gives 
$$\underline{w}(t,x) \leq \overline{w}(t,x) \hbox{ for all  } (t,x)\in H_T, $$ 
as $\alpha_m(w)$, $\underline{w}(t,x)$ and $\overline{w}(t,x)$ are all continuous functions, and  $\alpha_m(w)$ is increasing in $w$. The proof of Lemma \ref{exappcom} is thus complete.
\end{proof}

Clearly, $\underline{w}\equiv 0$ is a subsolution of \eqref{exappx}.  On the other hand, by the condition \eqref{assumeg} on $g$ and the property \eqref{restriction} of $\alpha_m$, we can conclude that the function $\overline{w}(t)$ is a bounded supersolution of \eqref{exappx}, where 
\begin{equation}\label{deuperw}
\overline{w}(t)=\|\tilde{u}_0\|_{L^{\infty}(\R^N)} \me^{Kt} \, \hbox{ for } 0\leq t\leq T
\end{equation}
is the unique solution to  the problem 
$$\frac{d \overline{w}}{d t} =K \overline{w} \,\hbox{ for } 0<t<T, 
 \quad  \overline{w}(0)=\|\tilde{u}_0\|_{L^{\infty}(\R^N)}.$$
 We are ready to show the existence and uniqueness of classical solution to \eqref{exappx}.

\begin{lem}\label{exapplem}
Let $m\in \mathbb{N}$ be fixed and $\overline{w}(t)$ be given as in \eqref{deuperw}. Then problem \eqref{exappx} admits a unique solution $w_m\in C^{1,2}((0,T]\times \R^N)\cap C(H_T)\cap L^\infty(H_T)$. Moreover, 
\begin{equation}\label{3.24}
\mbox{$0\leq w_m(t,x) \leq \overline{w}(t)$ for $(t,x)\in H_T$,}
\end{equation}
 and for any  $R>0$, there exists a positive constant $\tilde{C}_1=\tilde{C}_1(T,R)$ independent of $m$ such that
\begin{equation}\label{exh1n}
\| w_m\|_{H^1([0,T]\times B_R)} \leq \tilde{C}_1\; \hbox{ for all } m \in \mathbb{N}.
\end{equation}
\end{lem}

\begin{proof}
The existence part must be a known result, but we failed  to find a proof in the literature. So we include a proof here.
Note that the uniqueness part follows directly from Lemma \ref{exappcom}. 

For each $k\in\mathbb{N}$, let $B_{k} \in\R^N$ be the ball with center at $0$ and radius $k$.  Consider the following initial boundary value problem 
\begin{equation}\label{exapptru} 
\left \{ \begin{array}{ll} 
\displaystyle\medskip\frac{\partial \alpha_m
(v)}{\partial t}-d \Delta v=g(x, v) \;\;\;& \mbox{in $(0,T]\times B_k$},\\
\displaystyle\medskip v(t,x)=0 \;\;\; &\mbox{on $(0,T]\times \partial B_k$},\\
 \displaystyle v (0,x)={\tilde u}_0 (x) \;\;\; &\mbox{in $B_k$}.
\end{array} \right.
\end{equation}
It is well known that, for each $k\in\mathbb{N}$, problem \eqref{exapptru} admits a unique solution $v_k \in C([0,T]\times {B_k})\cap C^{1,2}((0,T]\times \overline B_k)$ (see e.g., \cite{LSU}). Moreover, it follows
from the comparison result given in \cite[Lemma 3.2]{DG2} that 
\begin{equation*}
0\leq v_k(t,x) \leq  \overline{w}(t)  \hbox{ for all } (t,x)\in [0,T]\times B_k.
\end{equation*}
It thus follows that 
\begin{equation*}
0\leq v_k(t,x) \leq  \tilde{C}_2:=\|\tilde{u}_0\|_{L^{\infty}(\R^N)} \me^{KT}   \hbox{ for all } (t,x)\in [0,T]\times B_k.
\end{equation*}
The comparison argument also gives $v_k(t,x)\leq v_{k+1}(t,x)$ for $t\in [0, T]$ and $x\in B_k$. Thus
\[
v(t,x):=\lim_{k\to\infty} v_k(t,x) \mbox{ exists for every $x\in \R^N$ and $t\in [0, T]$},
\]
and $0\leq v(t,x)\leq \overline w(t)$. By a standard regularity consideration, one sees that $v$ solves \eqref{exappx}, and \eqref{3.24} holds. To mark its dependence on $m$, we will denote $v=w_m$.

Next, we show the estimate \eqref{exh1n}. Clearly, we have
\begin{equation}\label{exbound}
0\leq w_m(t,x)\leq \tilde{C}_2 \hbox{ for } (t,x)\in H_T.  
\end{equation}
For any fixed  $R>0$,  let $\xi (x)$ be a smooth function such that
$$0 \leq \xi \leq 1 \hbox{ in } \R^N,\; \;\; \xi \equiv 1 \;\; \mbox{in $B_{R}$},
\,\hbox{ and }\, \xi \equiv 0 \;\; \mbox{in $\R^N \backslash B_{2R}$}.$$ 
If we multiply the differential equation in \eqref{exappx}  by $\alpha_m(w_m)\xi^2$, then by following the proof of \cite[Lemma 5.1]{DG2}, we obtain 
$$\int_0^T \int_{B_{R}} \big|\nabla_x w_m  \big|^2\leq \tilde{C}_3,  $$
for some positive constant $\tilde{C}_3=\tilde{C}_3(R,T)$ independent of $m$.  
Furthermore, if we multiply the differential equation in \eqref{exappx}  by $\frac{\partial w_m}{\partial t} \xi^2$ and estimate the resulting equation over $[0,T]\times B_{2R}$, then the proof of \cite[Lemma 5.2]{DG2} implies that there exists a positive constant $\tilde{C}_4=\tilde{C}_4(R,T)$ such that 
$$\int_0^T \int_{B_{R}} \Big|\frac{\partial_t w_m}{\partial t}  \Big|^2\leq \tilde{C}_4.  $$
Therefore, combining the above, we obtain \eqref{exh1n}.
The proof of Lemma \ref{exapplem} is now complete.
\end{proof}

To complete the proof of Theorem \ref{existence}, we shall need the following estimate.

\begin{lem}\label{lmesgw} 
For any $R>0$, there is a positive constant $\tilde{C}_5$ depending on $R$ and $T$ but
independent of $m$ such that  
\begin{equation*} 
\int_{B_{R/2}} \Big|\nabla_x w_m (\sigma,x)\Big|^2 dx \leq \tilde{C}_5,
\; \hbox{ for all } m \in \mathbb{N}, \;\;\sigma \in [0, T].
\end{equation*}
\end{lem}

\begin{proof}  
Let $\xi (x)$ be a smooth function such that
$$0 \leq \xi \leq 1 \hbox{ in } \R^N,\; \;\; \xi \equiv 1 \;\; \mbox{in $B_{R/2}$},
\,\hbox{ and }\, \xi \equiv 0 \;\; \mbox{in $\R^N \backslash B_{R}$}.$$ 
We multiply both sides of the differential equation in
\eqref{exappx} by $\frac{\partial w_m}{\partial t} \xi^2$ and then integrate
the resulting equation over $[0, \sigma] \times B_{R}$.  After suitable
integration by parts, we obtain
\begin{equation} \label{3.9} 
\left.\begin{array}{lll} 
 \displaystyle\medskip
\int_0^\sigma \int_{B_{R}} \alpha_m' (w_m)
\Big(\frac{\partial w_m}{\partial t} \Big)^2 \xi^2 dx dt &+& \displaystyle\medskip d
\int_0^\sigma \int_{B_{R}} \nabla_x w_m \cdot \nabla_x \Big(
\frac{\partial w_m}{\partial t} \xi^2 \Big) dx dt \\
 \displaystyle & & =\displaystyle \int_0^\sigma \int_{B_{R}} g(x, w_m) \frac{\partial w_m}{\partial t} \xi^2 dxdt.\end{array}\right.
\end{equation}
Moreover, using the assumption \eqref{assumeg} on $g$ and the estimate \eqref{exbound}, we have
\begin{equation*}
\int_0^\sigma \int_{B_{R}} g(x, w_m) \frac{\partial w_m}{\partial t} \xi^2 dxdt \leq  \frac{1}{4}\int_0^\sigma \int_{B_{R}}\Big(\frac{\partial w_m}{\partial t}\Big)^2 \xi^2 dxdt \,+\, \tilde{C}_6(T, R).
\end{equation*}
On the other hand, making use of the estimate \eqref{exh1n}, we deduce that
\begin{equation*}
\left.\begin{array}{ll}
&  \displaystyle\medskip d\int_0^\sigma \int_{B_{R}} \nabla_x w_m \cdot \nabla_x \Big(
\frac{\partial w_m}{\partial t} \xi^2 \Big) dx dt\\
=& d\displaystyle\medskip \int_0^\sigma \int_{B_{R}} \xi^2 \nabla_x w_m \cdot \nabla_x
\Big(\frac{\partial w_m}{\partial t} \Big) dx dt+2d\int_0^\sigma
\int_{B_{R}} \xi \frac{\partial w_m}{\partial
t} \nabla_x w_m \cdot \nabla_x \xi dx dt\\
\geq & \displaystyle\medskip \frac{d}{2} \int_{B_{R}} \big|\nabla_x w_m (\sigma, x)\big|^2\xi^2 (x) dx -\frac{d}{2} \int_{B_{R}} \big|\nabla_x {\tilde u}_0 (x)\big|^2 \xi^2 (x) dx\\
&\displaystyle\medskip \qquad-16d^2\int_0^\sigma \int_{B_{R}}  \big(\nabla_x w_m \cdot \nabla_x \xi\big)^2 dx dt-\frac{1}{4}\int_0^\sigma \int_{B_{R}}\Big(\frac{\partial w_m}{\partial t}\Big)^2 \xi^2 dxdt\\
\geq & \displaystyle\medskip \frac{d}{2} \int_{B_{R}} \big|\nabla_x w_m (\sigma, x)\big|^2\xi^2 (x) dx -\frac{d}{2} \int_{B_{R}} \big|\nabla_x {\tilde u}_0 (x)\big|^2 \xi^2 (x) dx\\
&\displaystyle\qquad-\frac{1}{4}\int_0^\sigma \int_{B_{R}}\Big(\frac{\partial w_m}{\partial t}\Big)^2 \xi^2 dxdt \,-\, \tilde{C}_7(T, R).
\end{array}\right.
\end{equation*}
Substituting the above estimates into \eqref{3.9}, and recalling that $\alpha_m'(w) \geq 1$ from \eqref{restriction}, we obtain
\begin{equation*}
\left.\begin{array}{ll} 
& \displaystyle\medskip \frac{1}{2}\int_0^\sigma \int_{B_{R}} \Big(\frac{\partial w_m}{\partial t} \Big)^2 \xi^2 dx dt + \frac{d}{2}\int_{B_{R}} \big|\nabla_x w_m (\sigma, x)\big|^2\xi^2 (x) dx\\
\leq & \displaystyle \frac{d}{2} \int_{B_{R}} \big|\nabla_x {\tilde u}_0 (x)\big|^2 \xi^2 (x) dx + \tilde{C}_6(T, R)+\tilde{C}_7(T, R).
\end{array}\right.
\end{equation*}
Let 
$$\tilde{C}_5:=  d \int_{B_{R}} \big|\nabla_x {\tilde u}_0 (x)\big|^2 dx +2 \tilde{C}_6(T, R)+2\tilde{C}_7(T, R).
 $$
It finally follows that
$$\int_{B_{R/2}} \big|\nabla_x w_m (\sigma,x)\big|^2 dx \leq \int_{B_{R}} \big|\nabla_x w_m (\sigma,x)\big|^2 \xi^2(x) dx
\leq \tilde{C}_5 \;\; \hbox{ for all } m \in \mathbb{N}.$$ 
This completes the proof. 
\end{proof}

We are now ready to give the proof of Theorem \ref{existence}.

\begin{proof} [Proof of Theorem \ref{existence}] 
In what follows we will select various subsequences from $\{w_m\}$
and, to avoid inundation by subscripts, we will always denote the
subsequence again by $\{w_m\}$. By the estimate \eqref{exh1n}, and by
Rellich's Lemma and a standard diagonal argument,  there is a subsequence of $\{w_m\}$ (denoted by itself) and a
function $w \in H^1_{loc} (H_T)$ such that, for any $R>0$,
\begin{equation*}
w_m \to w \;\; \mbox{weakly in $H^1 ([0,T]\times B_R)$
and strongly in $L^2 ([0,T]\times B_R)$}\, \hbox{ as }  m \to \infty.
\end{equation*}
In particular, $w_m \to w$ as $m\to\infty$ and $w \geq 0$ almost everywhere in
$H_T$. Moreover, in view of \eqref{exbound}, we have $0 \leq w \leq \tilde{C}_2$ in $H_T$, and hence $w\in H^1_{loc} (H_T)\cap L^{\infty}(H_T)$.
Furthermore, Lemma \ref{lmesgw} implies that
\begin{equation*}
\int_{B_{R/2}} |\nabla_x w(t,x)|^2 dx \leq \tilde{C}_5 \;\;
\mbox{for a.e. $t \in [0,T]$},
\end{equation*}
since the set
$$\Big \{v \in H^1([0,T]\times B_{R/2}): \; \int_{B_{R/2}} \big|\nabla_x
v(t,x)\big|^2 dx \leq \tilde{C}_5 \;\; \mbox{for a.e. $t \in [0,T]$} \Big \}$$
is closed and convex in $H^1([0,T]\times B_{R/2})$, and such sets are
closed under the weak limit.

With the above preparations, and noting that the test function $\phi$ in Definition \ref{defweak} has compact support, we can follow the same lines as those used in the proof of \cite[Theorem 3.1]{DG2} to verify that $w$ is a weak solution of \eqref{eqmain}  over $H_T$. 
We do not repeat the details here.    
The proof of Theorem \ref{existence} is thereby complete.
\end{proof}

Note that since $T>0$ is arbitrary in Theorem \ref{existence}, the weak solution $u(t,x)$ of \eqref{eqmain} given over $H_T$ can be extended to all $t>0$, and it is unique due to Theorem \ref{unique}.

Next, we present a comparison principle which will be used frequently in the subsequent sections. Suppose that $\mu$ and $\hat{\mu}$ are two positive constants, $g$ and ${\hat g}$ both satisfy the assumption \eqref{assumeg}, $\Omega_0$ and ${\hat \Omega}_0$ are smooth domains in $\R^N$, $u_0$ satisfies \eqref{assumeu0} and ${\hat u}_0$ satisfies \eqref{assumeu0} with $\Omega_0$ replaced by ${\hat \Omega}_0$. Let $u$ be a weak subsolution of \eqref{eqmain} corresponding to  $(\Omega_0, u_0, g,\mu)$, and ${\hat u}$ be a weak supersolution corresponding to $({\hat \Omega}_0, {\hat u}_0, {\hat g},\hat{\mu})$, respectively.

\begin{thm} \label{thmcomp} 
Suppose that $\Omega_0 \subset {\hat \Omega}_0$, $u_0
\leq {\hat u}_0$, $\mu\leq \hat{\mu}$ and $g \leq {\hat g}$. Then $u \leq {\hat u}$ a.e. in $H_T$.
\end{thm}

\begin{proof}
 Let $w$ and $\hat{w}$ be the weak solutions of \eqref{eqmain} corresponding to $(\Omega_0, u_0, g,\mu)$ and $({\hat \Omega}_0, {\hat u}_0, {\hat g},\mu)$ over $H_T$, respectively. As a direct application of Theorem \ref{unique}, we have
 \begin{equation}\label{3.14}
 u\leq w, \quad  \hat{u}\geq \hat{w} \,\,\, a.e. \hbox{ in } H_T.
 \end{equation}

On the other hand, let $w_m$ be the solution of problem \eqref{exappx}, and 
${\hat w}_m$ be the solution determined by \eqref{exappx}  with
reaction term ${\hat g}$, and the initial
function is obtained by extending ${\hat u}_0$ to $\R^N$. By the comparison result Lemma \ref{exappcom}, we
clearly have $w_m \leq {\hat w}_m$ in $ H_T$. Furthermore, from the proof of Theorem
\ref{existence} and the uniqueness of the weak solution, we know that
$$w_m \to w, \quad  \hat{w}_m \to \hat{w} \,\,\, a.e. \hbox{ in }  H_T \hbox{ as } m\to\infty,$$
and hence $w\leq \hat{w}$ a.e. in $H_T$. This together with \eqref{3.14} completes the proof.
\end{proof}

The next result is useful for constructing supersolutions and subsolutions to \eqref{eqmain}.

\begin{thm}\label{super-sub}
Assume that
 there exists a $C^1$ function $\Phi$ over ${H_T}$
 such that, with
  $$\Omega (t): = \Big\{x \in \R^N: \; \Phi
 (t,x)<0\Big\},
 $$ 
 one has
 \[
 \mbox{
 $\Omega (0)=\Omega_0$ and } \; |\nabla_x \Phi| \neq 0 \;\;\; \mbox{on\; $ \partial
\Omega (t)$}.
\]
Suppose that  $w(t,x)$ and $\nabla_x w(t,x)$ are continuous in $\bigcup_{0 \leq
t\leq T} {\overline {\Omega (t)}}$, $\nabla_x^2 w(t,x)$, $w_t(t,x)$ are
continuous in $\bigcup_{0<t\leq T} \Omega (t)$. Then $w(t,x)$ (extended by 0 for $x\not\in\Omega(t)$) is a weak supersolution of \eqref{eqmain} for $t\in (0, T]$, provided that
\begin{equation}\label{sup}\left\{
\begin{array}{ll}
w_t-d\Delta w\geq g(x, w),\; w>0 & \mbox{ for } x\in\Omega(t),\; 0<t\leq T,\smallskip
\\
w=0,\; \Phi_t\leq \mu \nabla_x u\cdot\nabla_x\Phi &\mbox{ for } x\in \partial\Omega(t),\; 0<t\leq T,\smallskip
\\
w(0,x)\geq u_0(x) & \mbox{ for } x\in\Omega_0.
\end{array}\right.
\end{equation}
If all the inequalities in \eqref{sup} are reversed, then $w$ is a weak subsolution to \eqref{eqmain}.
\end{thm}
\begin{proof} We only prove the case for weak supersolution, as the proof for the weak subsolution case is similar. We need to show that
\begin{equation}\label{eqweaksuper} 
\left.\begin{array}{ll}
\displaystyle\medskip \int_0^T \int_{\R ^N} \Big[d \nabla_x w \cdot \nabla_x
\phi-\alpha (w) \phi_t\Big] dx dt-\int_{\R ^N} &\!\!\!\!\!\!\alpha (\tilde u_0(x)) \phi (0,x) dx\\
&\displaystyle \geq \int_0^T \int_{\R ^N} g(x,w) \phi dxdt
\end{array}\right.
\end{equation}
for every nonnegative function $\phi \in C^{\infty}(H_T)$ with compact support and satisfying
$\phi=0$ on $\{T\} \times \R^N$.

  For each test function $\phi$ as described above,  we use 
the divergence theorem to calculate the following integral 
$$\int _{0}^{T}\int_{\Omega(t)} {\rm div} \Psi \,dxdt \;\mbox{ with }
\Psi(t,x)=(\phi(t,x),0,\cdots,0)\in\R^{N+1},
$$ 
and obtain, with $S:=\bigcup_{0<t<T}\partial\Omega(t)$,
 \begin{equation*}
\left.\begin{array}{ll}
\displaystyle\medskip  \int _{0}^{T}\int_{\Omega(t)} \phi_t\; dxdt 
 \!\! &=\displaystyle\medskip -\int_{S} \phi\; \frac{\Phi_t}{\big|(\Phi_t, \nabla_x \Phi)\big|}d\sigma - \int_{{\Omega}(0)} \phi(0,x)\; dxdt  \\
 &= \displaystyle\medskip -\int_{0}^T\int_{\partial\Omega(t)} \phi\; \frac{\Phi_t}{\big|\nabla_x \Phi\big|}d S_xdt - \int_{{\Omega}(0)} \phi(0,x) dxdt.
 \end{array}\right.
\end{equation*}

Since the unit outward normal of $\Omega(t)$ at $x\in\partial\Omega(t)$ is given by 
$\nu=\nu(t,x):=\nabla_x\Phi(t,x)/|\nabla_x\Phi(t,x)|$, by \eqref{sup} we have
\begin{equation*}
\frac{\Phi_t}{|\nabla_x\Phi|}\leq \mu \nabla_x w\cdot \nu \mbox{ for } x\in\partial\Omega(t),
\end{equation*}
and hence
\begin{equation}\label{inteul0}
-\frac{d}{\mu}\int _{0}^{T}\int_{{\Omega}(t)} \phi_t\; dxdt  \leq   d\int_{0}^T\int_{\partial\Omega(t)} \phi\; \nabla_x w\cdot\nu d S_xdt +\frac{d}{\mu} \int_{\Omega(0)} \phi(0,x) dxdt. 
\end{equation}

On the other hand, we multiply both sides of the first inequality in \eqref{sup} by $\phi$ and integrate the resulting inequality over $\bigcup_{0<t<T} {\Omega}(t)$, making use of  integration by parts and the last inequality in \eqref{sup}, and obtain
 \begin{equation*}
\left.\begin{array}{ll}
\displaystyle\medskip  \int _{0}^{T}\int_{{\Omega}(t)} \big[\phi_t w-d\nabla_x \phi \nabla_x w\big]dxdt &+\,\,\displaystyle\medskip \int _{0}^{T}\int_{{\Omega}(t)} g(x, w) \phi dxdt\\
&\displaystyle\medskip\leq \,\, -d\int_{0}^T\int_{\partial\Omega(t)} \phi \nabla_xw\cdot\nu d S_xdt 
- \int_{{\Omega}(0)}u_0(x)\phi(0,x) dxdt.
 \end{array}\right.
\end{equation*} 
By this inequality, \eqref{inteul0}, the fact that $w(t,x)=0$ for $x\not\in\Omega(t)$, and the definition of $\alpha$,
we immediately obtain \eqref{eqweaksuper}. 
\end{proof}

We now consider the asymptotic behavior of the weak solution to \eqref{eqmain} as $\mu \to \infty$. To emphasize its dependence on $\mu$, we denote by $u_\mu$ the unique weak solution, and denote $\Omega_\mu (t)=\{x: u_\mu (t,x)>0\}$. 

\begin{thm} \label{asymu} 
Let $u_\mu$ be the unique weak solution to problem \eqref{eqmain}. Then
\begin{equation*} 
\lim_{\mu \to \infty} \Omega_\mu (t)=\R^N, \;\; \forall \,t>0,
\end{equation*}
and
\begin{equation*}
u_\mu \to U \;\; \mbox{in $C_{loc}^{\frac{1+\nu_0}{2},
1+\nu_0} ((0, \infty) \times \R^N)$ as $\mu \to \infty$},
\end{equation*}
where $\nu_0$ can be any number in $(0,1)$ and $U(t,x)$ is the
unique solution of the Cauchy problem
\begin{equation}\label{cauchyp} 
\left \{ \begin{array}{ll}
\displaystyle\medskip U_t-d \Delta U=g(x,U) \;\;&\mbox{in $(0, \infty) \times \R^N$},\\
U(0,x)={\tilde u}_0 (x) \;\;& \mbox{in $\R^N$}.
\end{array} \right.
\end{equation}
Moreover, $u_\mu (t,x) \leq U(t,x)$ for all $t>0$ and $x \in \Omega_\mu (t)$.
\end{thm}

\begin{proof}
Making use of Theorem \ref{thmcomp}, the proof is almost identical to that of  \cite[Theorem 5.4]{DG2}, and we omit the details.
\end{proof}

\begin{rem}
Although our main interest of this paper is for the case $\Omega_0\subset\R^N$ with $N\geq 2$, it is easily seen 
from their proofs that all the results in this section remain valid when $N=1$.
\end{rem}


\section{Spreading profile of the Fisher-KPP equation}\label{sec4}

In this section, we study the long-time behavior of the weak solution to problem \eqref{eqfrfisher}, for some special unbounded
 $\Omega_0$. More precisely, we assume that there exist $\phi\in (0, \pi)$ and $\xi_1>\xi_2$ such that \eqref{Omega0} holds, namely
\begin{equation}\label{outscone}
\Lambda^\phi+\xi_1e_N \subset {\Omega_0} \subset \Lambda^\phi+\xi_2e_N.
\end{equation}

Our first result on the long-time behavior of \eqref{eqfrfisher} is the following theorem.

\begin{thm}\label{thm6.1} 
Let $u(t,x)$ be the unique weak solution of problem \eqref{eqfrfisher} with $\Omega_0$ satisfying  \eqref{outscone}.
 Denote $\Omega(t)=\big\{x:\, u(t,x)>0\big\}$. Then
\begin{equation}\label{convtorn}
\lim_{t \to \infty} \Omega (t)=\R^N,
\end{equation}
and
\begin{equation}\label{convtoss}
 \lim_{t \to \infty} u(t,x)=\frac ab \;\;
\mbox{locally uniformly in $x \in \R^N$}.
\end{equation}
\end{thm}

\begin{proof}
We first prove  \eqref{convtorn}. Since $\Omega_0$ satisfies \eqref{outscone}, we can find a ball $B_{r_0}(x_0)\subset \Omega_0$ with radius 
\begin{equation}\label{radius}
r_0\geq R^*:= \sqrt{\frac{d}{a}\lambda_1},
\end{equation}
 where $\lambda_1$ is the first eigenvalue of the following eigenvalue problem
 $$-\Delta \phi=\lambda \phi \hbox{ in } B_{1}(0);\quad \phi=0 \hbox{ on } \partial B_1(0).$$
 We then choose a $C^2$ radial function $v_0(r)$ ($r=|x-x_0|$) satisfying 
$$v_0(|x-x_0|)\leq u_0(x) \hbox{ for } |x-x_0|<r_0, \,\,\, v_0(r)>0 \hbox{ for } r\in [0,r_0)\,\hbox{ and }\, v_0'(0)=v_0(r_0)=0,$$
and consider the following radially symmetric problem
\begin{equation*}
\left \{ \begin{array}{ll} 
\displaystyle\medskip v_t-d \Delta v=av-b v^2, \;\;&t>0, \; 0<r<k(t), \\
\displaystyle\medskip v_r (t,0)=0, \;\; v(t,k(t))=0, \;\; &t>0,\\
\displaystyle\medskip k'(t)=-\mu v_r (t, k(t)), \;\; &t>0,\\
\displaystyle  k(0)=r_0, \;\;\; v(0,r)=v_0 (r), \;\; & 0 \leq r \leq r_0.
\end{array} \right.
\end{equation*}
It follows from \cite[Theorem 2.1]{DG1} that this problem admits a (unique)
classical solution $(v(t,r), k(t))$ defined for all $t>0$ such that $k'(t)>0$, $ v(t,r)>0$ for $0 \leq r<k(t)$, $t>0$.
 Moreover, by \cite[Theorem 2.5]{DG1}, we have 
 \begin{equation}\label{setbelow} 
 \lim_{t\to\infty}k(t)=\infty.
 \end{equation} 
 
Denote 
$$\mathcal{G} (t)=\Big\{x \in \R^N: \; |x-x_0|<k(t)\Big\},$$
and
$$ V(t,x)=v(t,|x-x_0|).  $$
We also extend $V(t,\cdot)$ to be zero outside $\mathcal{G} (t)$. Clearly, for any given $T>0$, $(V,\mathcal G)$ 
 is a classical solution of the free boundary problem
 \begin{equation}\label{eq6.2} 
\left \{ \begin{array} {ll} \displaystyle\medskip  u_t-d \Delta u=au-bu^2 \;\; &
\mbox{for $x \in \mathcal{G} (t), \; 0<t\leq T$},\\
\displaystyle\medskip u=0,\; u_t=\mu|\nabla_x u|^2 \;\;&\mbox{for $x \in \partial\mathcal{G}(t), \; 0<t\leq T$}, \\
u(0,x)=v_0 (|x-x_0|) \;\; &\hbox{for } x \in \mathcal{G} (0). 
\end{array} \right.
\end{equation}
By Theorems \ref{weakclassic} and \ref{unique}, $V$ is the unique weak solution of $\eqref{eq6.2}$ over $H_T$. It then follows from Theorem \ref{thmcomp} that $u \geq V $ in  $H_T$, and hence by the arbitrariness of $T$, we have
$ \mathcal{G} (t)\subset \Omega (t)$ for all $t\geq 0$. This together with \eqref{setbelow} implies \eqref{convtorn}.
  
It remains to prove  \eqref{convtoss}. 
We claim that 
\begin{equation}\label{auxicom} 
{V}(t,x)\leq u(t,x)\leq {U}(t,x) \hbox{  in }   [0,\infty)\times \R^N,
\end{equation}
where $U(t,x)$ is the unique solution to the Cauchy  problem \eqref{cauchyp}.
Indeed, the first inequality in \eqref{auxicom} follows directly from Theorem \ref{thmcomp}, while the second inequality is a consequence of Theorem \ref{asymu}.

Applying \cite[Theorem 6.2]{DG2} to the equation of ${V}$, we obtain 
 \begin{equation}\label{stablefb}
 \lim_{t \to \infty} {V}(t,x)=\frac ab \;\;\mbox{locally uniformly in $x \in \R^N$}.
\end{equation}
On the other hand, it is well known  that  
\begin{equation}\label{stablefa}
 \lim_{t \to \infty} {U}(t,x)=\frac ab \;\;
\mbox{locally uniformly in $x \in \R^N$}.
\end{equation}  
By \eqref{auxicom}, \eqref{stablefb} and \eqref{stablefa}, we immediately obtain \eqref{convtoss}, and the proof of Theorem \ref{thm6.1} is thus complete.
\end{proof}

Next we examine  the long-time profile of the free boundary $\partial \Omega(t)$. 
For convenience, we first recall the following result from \cite{BDK}. 
\begin{prop}\label{semiwave}
Let $d>0$, $a>0$ and $b>0$ be given constants. For any $k\in [0, 2\sqrt{ad})$, the problem 
\begin{equation}\label{eqwave}
-dZ''+kZ'=aZ-bZ^2 \hbox{ in } (0,\infty),\quad Z(0)=0
\end{equation}
admits a unique positive solution $Z=Z_k$, and it satisfies $Z_k(r)\to a/b$ as $r\to\infty$. Moreover, $Z'_k(r)>0$ for $r\geq 0$, $Z_{k_1}'(0)>Z_{k_2}'(0)$, $Z_{k_1}(r)>Z_{k_2}(r)$ for $r>0$ and $k_1<k_2$, and for each $\mu>0$, there is a unique $c_*=c_*(\mu,a,b,d)$ such that $\mu Z'_{k}(0)=k$ when $k=c_*$. Furthermore, $c_*(\mu,a,b,d)$ depends continuously on its arguments, is increasing in $\mu$ and $\lim_{\mu\to\infty}c_*(\mu,a,b,d)=2\sqrt{ad}$.
\end{prop}

In our discussion below, since $a$, $b$, $d$ and $\mu$ are always fixed, we use  $c_*$ to denote $c_*(\mu,a,b,d)$. 

\begin{thm}\label{spreadspeed}
Let $u(t,x)$ be the unique weak solution of problem \eqref{eqfrfisher} with $\Omega_0$ satisfying \eqref{outscone} 
for some $\phi\in (\pi/2, \pi)$, and $u_0$ satisfying \eqref{assumeu0} and
\begin{equation}\label{addau0}
\liminf_{d(x,\partial \Omega_0)\to\infty} u_0(x)>0.
\end{equation}
Denote $\Omega(t)=\big\{x:\, u(t,x)>0\big\}$. Then for any $\epsilon>0$, there exists $T=T(\epsilon)>0$ such that for all $t>T$
we have
\begin{equation}\label{Omega(t)}
\Lambda^\phi-\left(\frac{c_*}{\sin\phi}-\epsilon\right)t\,{e}_N\subset \Omega(t)\subset \Lambda^\phi-\left(\frac{c_*}{\sin\phi}+\epsilon\right)t\,{e}_N.
\end{equation}
 Moreover, 
\begin{equation}\label{outconebehu}
\lim_{t\to\infty}\left[{\sup}_{x\in   \Lambda^\phi-\big(\frac{c_*}{\sin\phi}-\epsilon\big)t\,{e}_N} \Big|u(t,x)-\frac{a}{b} \Big|\right]=0.
\end{equation}
\end{thm}

We prove this theorem by a series of lemmas. We first show 
\begin{equation}\label{Omega(t)-lb}
\Lambda^\phi-\left(\frac{c_*}{\sin\phi}-\epsilon\right)t\,{e}_N\subset \Omega(t) \mbox{ for all large $t$. }
\end{equation}
By assumption \eqref{addau0}, we can find a one-dimensional function $w_0\in C^2([0,\infty))$ such that 
\begin{equation}\label{lowu0}
0<w_0\leq \sigma_0:=\frac{1}{2}\liminf_{d(x,\partial \Omega_0)\to\infty } u_0(x) \hbox{ in } (0,\infty),\quad w_0(0)=0 \quad\hbox{and}\quad  w_0'>0 \hbox{ in } [0,\infty).
\end{equation}
Then we consider the following one space dimension free boundary problem   
\begin{equation}\label{eqlow} 
\left \{ \begin{array}{ll} 
\displaystyle\medskip w_t-d w_{yy}=aw-bw^2, \;\;&t>0, \; \rho(t)<y<\infty, \\
\displaystyle\medskip w(t,\rho(t))=0, \;\; &t>0,\\
\displaystyle\medskip \rho'(t)=-\mu w_y (t,\rho(t)), \;\; &t>0,\\
\displaystyle  \rho(0)=0, \;\;\; w(0,y)=w_0 (y), \;\; & 0 \leq y <\infty.
\end{array} \right.
\end{equation}
It follows from \cite[Theorem 2.11]{DDL} that, \eqref{eqlow} admits a (unique)
classical solution $(w(t,y), \rho(t))$ defined for all $t>0$ and
$\rho'(t)<0$, $w(t,y)>0$ for $\rho(t) \leq y<\infty$, $t>0$.
Moreover, we have the following comparison result.

\begin{lem}\label{onedimcom}
 For any given $\tilde{T}\in (0,\infty)$, suppose  $\tilde{\rho}\in C^1\big([0,\tilde{T}]\big)$ and  $\tilde{w}\in  C^{1,2}\big(\tilde{D}_{\tilde{T}}\big)$ with $\tilde{D}_{\tilde{T}}=\big\{(t,y): \, 0\leq t\leq \tilde{T},\, \tilde{\rho}(t)\leq y<\infty\big\}$. If 
\begin{equation*}\left\{\baa{ll}
\displaystyle\medskip \tilde{w}_t -d\tilde{w}_{yy}\geq\tilde{w}(a-b \tilde{w}) ,& 0<t\leq \tilde{T},\,\,\,\tilde{\rho}(t)<y<\infty,\vspace{3pt}\\
\tilde{w}(t,\tilde{\rho}(t))= 0,\,\,\,\, \tilde{\rho}'(t)\leq -\mu \tilde{w}_y(t,\tilde{\rho}(t)),&0<t\leq \tilde{T},\eaa\right.
\end{equation*} 
and
$$\tilde{\rho}(0) \leq 0 \quad\hbox{and}\quad w_0(y)\leq \tilde{w}(0,y)\,\hbox{ in } \, [0,\infty),$$
 then the solution $(w,\rho)$ of problem \eqref{eqlow} satisfies
$$ \tilde{\rho}(t)\leq \rho(t) \,\hbox{ in }\, (0,\tilde{T}]\,\hbox{ and }\,  w(t,y)\leq \tilde{w}(t,y)\,\hbox{ for } \, 0<t\leq \tilde{T},\,\, \rho(t)\leq y <\infty.$$ 
\end{lem}
\begin{proof}
If for any $0\leq t\leq \tilde{T}$, we extend $w(t,y)$ (resp. $\tilde{w}(t,y)$) to be zero for $y< \rho(t)$ (resp. $y< \tilde{\rho}(t)$), then it is easily checked that $w$ (reps. $\tilde{w}$) is a weak solution (resp. weak supersolution) of the free boundary problem induced from \eqref{eqlow} over $H_{\tilde{T}}$, and hence the desired comparison result follows from Theorem \ref{thmcomp}. (One could also prove the result directly along the lines of \cite{DL}.)
\end{proof}

We next show the following estimate of $(w,\rho)$.

\begin{lem}\label{eslowra}
For any $\epsilon>0$, there exists $T_1=T_1(\epsilon)>0$ such that 
\begin{equation}\label{eslowrap1}
\rho(t) \leq  -\Big(c_{*}-\frac{2}{3}\epsilon\Big) t\,\hbox{ for all } t\geq T_1,  
\end{equation}
and
\begin{equation}\label{eslowrap2}
\liminf_{t\to\infty}\left[\inf_{y\geq -(c_{*}-\frac 23\epsilon) t} w(t,y)\right]\geq \frac{a}{b}.
\end{equation}
\end{lem}

\begin{proof}
The proof follows from \cite[Theorem 4.2]{DL} with some modifications. For the sake of completeness,
and also for the convenience of later applications, we include the details below. 

We first claim that for any given small $\delta>0$, there exists $t_1=t_1(\delta)>0$ such that 
\begin{equation}\label{eqlargt1}
w(t_1,y)\geq \frac{a-\delta}{b+\delta}  \hbox{ for all } y\geq 0.  
\end{equation}
Indeed, applying the proof of Theorem \ref{thm6.1} to the one-dimensional problem \eqref{eqlow}, we easily
obtain $\lim_{t\to\infty} \rho(t)=-\infty$ and   
$$ \lim_{t \to \infty} w(t,y)=\frac{a}{b} \;\;\mbox{locally uniformly for $y \in \R$}.$$
Since $w_0(y)$ is nondecreasing in $y\in[0,\infty)$, it follows from the comparison result stated in Lemma \ref{onedimcom} and the uniqueness of solution to problem \eqref{eqlow} that, for any fixed $t>0$, $w(t,y)$ is nondecreasing in $y\in [\rho(t),\infty)$. We thus obtain 
$$ \liminf_{t \to \infty} w(t,y)\geq \frac{a}{b} \;\;\mbox{uniformly for $y \geq 0 $},$$
which clearly implies \eqref{eqlargt1}. 

Next, we construct a subsolution to problem \eqref{eqlow}. To do this, we need a few more notations. For any small $\delta>0$, denote $$c_{\delta}:=c_*(\mu,a-\delta,b+\delta,d)$$
and denote by $Z_{\delta}(r)$ the solution of \eqref{eqwave} with $k$, $a$, $b$ replaced by $c_{\delta}$, $a-\delta$, $b+\delta$, respectively. Set
$$
\eta(t)=\eta_\delta(t):=-(1-\delta)^2c_\delta t \,\hbox{ for } \,t>0,  
$$
and
$$
\underline{w}(t,y)=\underline{w}_\delta(t,y):=(1-\delta)^2 Z_{\delta}(y-\eta_\delta(t)) \,\hbox{ for } \,\eta_\delta(t)\leq y<\infty,\, t>0.
$$

It is straightforward to verify that 
$$\underline{w}(t,\eta(t))=0, \quad \eta'(t)=-\mu \underline{w}_y(t,\eta(t))\,\,\hbox{ for } \, t>0. $$
Moreover, since $Z'_{\delta}>0$ in $[0,\infty)$,  direct calculations yield 
\begin{equation*}
\left.\begin{array}{ll}
\displaystyle\medskip \underline{w}_t-d\underline{w}_{yy}&=(1-\delta)^{4}c_{\delta}Z_\delta'-d(1-\delta)^{2}Z_\delta''\\
&\displaystyle\medskip\leq (1-\delta)^{2}(Z_\delta'-d Z_\delta'')\\
&\displaystyle\medskip=(1-\delta)^{2}\big[(a-\delta)Z_\delta-(b+\delta) Z^2_\delta\big]\\
&\leq (a-\delta)\underline{w}-(b+\delta)\underline{w}^2
\end{array}\right.
\end{equation*}
for $\eta(t)<y<\infty$, $t>0$. By Proposition \ref{semiwave}, we have
$$ \underline{w}(0,y)=(1-\delta)^2Z_{\delta}({y})<(1-\delta)^2\frac{a-\delta}{b+\delta} \,\,\hbox{ for }  0\leq y <\infty.$$
This together with \eqref{eqlargt1} implies   
$$w(t_1,y) \geq \underline{w}(0,y) \,\hbox{ for }  0\leq y<\infty.  $$
It then follows from Lemma \ref{onedimcom} that 
\begin{equation}\label{comrelow}
\rho(t+t_1) \leq \eta(t), \quad  w(t+t_1,y)\geq \underline{w}(t,y) \,\hbox{ for } \,  \eta(t)\leq y<\infty,\,t>0. 
\end{equation}

Since $\lim_{\delta\to 0} (1-\delta)^2c_{\delta}=c_*:=c_*(\mu,a,b,d)$, for any $\epsilon>0$, we can find some $\delta_\epsilon\in (0,\epsilon)$ such that 
\begin{equation}\label{chdelta}
\big|(1-\delta_{\epsilon})^2c_{\delta_{\epsilon}}- c_*  \big|\leq \epsilon/2.   
\end{equation}
We now fix $\delta=\delta_{\epsilon}$ in $Z_{\delta}$, $\eta_\delta$ and $t_1(\delta)$.
Then \eqref{comrelow} implies 
$$
\rho(t)\leq -(1-\delta_{\epsilon})^2c_{\delta_{\epsilon}} (t-t_1) \leq -\Big(c_*-\frac{2}{3}\epsilon\Big)t-\frac{\epsilon}{6}t +(1-\delta_{\epsilon})^2c_{\delta_{\epsilon}}t_1 \,\hbox{ for } t\geq t_1.
$$
Thus \eqref{eslowrap1} holds with $T_1=6\epsilon^{-1}(1-\delta_{\epsilon})^2c_{\delta_{\epsilon}}t_1$.

It remains to prove \eqref{eslowrap2}. With $\delta=\delta_\epsilon$ chosen as above, it follows from Proposition \ref{semiwave} that there exists $y_0>0$ sufficiently large such that 
\begin{equation}\label{ineqlowinf}
Z_{\delta_{\epsilon}}(y)\geq \frac{a-2\delta_{\epsilon}}{b+2\delta_{\epsilon}}  \,\hbox{ for all } \,y\geq y_0.  
\end{equation}
On the other hand, by \eqref{chdelta}, we clearly have 
$$y-\eta(t-t_1)\geq y+\Big(c_*-\frac{2}{3}\epsilon\Big)t+\frac{\epsilon}{6}t-(1-\delta_{\epsilon})^2c_{\delta_{\epsilon}}t_1  \,\hbox{ for } t\geq t_1.$$
Thus, if we choose $\tilde{t}_1=6\epsilon^{-1}[y_0+(1-\delta_{\epsilon})^2c_{\delta_{\epsilon}}t_1]$, then 
$$y-\eta(t-t_1)\geq y_0\,\hbox{ for all }   y\geq -\Big(c_*-\frac{2}{3}\epsilon\Big)t,\,t\geq \tilde{t}_1.$$
This together with \eqref{comrelow} and \eqref{ineqlowinf} implies 
$$w(t,y)\geq \underline{w}(t-t_1,y) \geq (1-\delta_{\epsilon})^2\frac{a-2\delta_{\epsilon}}{b+2\delta_{\epsilon}} \,\hbox{ for all} \, y\geq -\Big(c_*-\frac{2}{3}\epsilon\Big)t,\,t\geq \tilde{t}_1.$$
Since 
$$(1-\delta_{\epsilon})^2\frac{a-2\delta_{\epsilon}}{b+2\delta_{\epsilon}} \to \frac{a}{b} \hbox{ as } {\epsilon}\to 0, $$
this gives \eqref{eslowrap2}, and the proof of Lemma \ref{eslowra} is now complete.
\end{proof}

\begin{rem}
The conclusions in Lemma 3.5 can be considerably sharpened (though they are not needed in this paper). It is possible to modify the method of \cite{DMZ1} to show that, as $t\to\infty$,
\[
\rho(t)-c_*t\to C\in\R^1, \;\; \sup_{y\in[\rho(t),\infty)}\big|w(t,y)-Z_{c_*}(y-\rho(t))\big|\to 0,
\]
where $Z_{c_*}$ is given in Proposition 3.2. 
\end{rem}

\begin{lem}\label{lowbound}
Let $u(t,x)$ and $\Omega(t)$ be given in the statement of Theorem \ref{spreadspeed}. Then for any $\epsilon>0$, there exists $T_2=T_2(\epsilon)>0$ such that \eqref{Omega(t)-lb} holds
 for all $ t\geq T_2$,
and 
\begin{equation}\label{infubound}
\liminf_{t\to\infty}\left[{\inf}_{x\in   \Lambda^{\phi}-\big(\frac{c_*}{\sin\phi}-\epsilon\big)t\,{e}_N } u(t,x)\right]\geq \frac{a}{b}.
\end{equation}
\end{lem}

\begin{proof}
By the assumption \eqref{addau0}, there exists $r_1>0$ sufficiently large such that 
$$u_0(x) \geq \sigma_0:= \frac{1}{2}\liminf_{d(\tilde{x},\partial \Omega_0)\to\infty}u_0(\tilde{x}) \hbox{ for all } x\in \Omega_0 \hbox{ with } d(x,\partial \Omega_0)\geq r_1. $$
Then, due to the assumption \eqref{outscone}, we have 
\begin{equation}\label{initialcom} 
u_0(x) \geq  \sigma_0  \hbox{ for all } x\in  \Lambda^{\phi}+ \Big(\xi_1+\frac{r_1}{\sin\theta}\Big)e_N,
\end{equation}
 where $\theta=\pi-\phi$ (and so $\sin\theta=\sin\phi$).

Let $(w,\rho)$ be the unique solution to problem \eqref{eqlow} with initial function $w_0$ satisfying \eqref{lowu0}. For  convenience of
notation, we write
\[
z_1:=\Big(\xi_1+\frac{r_1}{\sin\theta}\Big)e_N\quad\hbox{and}\quad \Lambda_{z_1}:= \Lambda^{\phi}+z_1.
\]
For any fixed $z\in \partial \Lambda_{z_1}\backslash \{ z_1\}$, let $\nu_z$ be the inward unit normal vector of $\Lambda_{z_1}$ at $z$, and define
$$
\Omega_{z}(t)
=\Big\{x:\,\, x\cdot \nu_z \geq \rho(t)+r_1+ \xi_1\sin\theta\Big\}$$
($\Omega_{z}(0)$ is illustrated in Figure \ref{fighalfs}), and 
$$ w_z(t,x)=w(t,x\cdot\nu_z-r_1- \xi_1\sin\theta).$$
\begin{figure}[h]
\centering
\def\svgwidth{8cm}
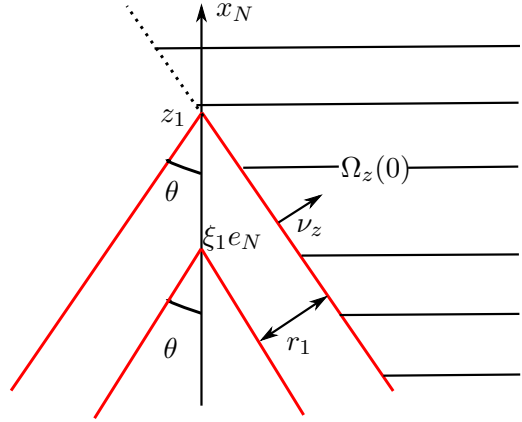
\caption{The domain $\Omega_z(0)$ with given $z\in \partial \Lambda_{z_1}\backslash \{ z_1\}$ }\label{fighalfs}
\end{figure}
 We also extend $w_z(t,\cdot)$ to be zero outside $\Omega_z(t)$ for $t\geq 0$. Clearly, $(w_z,\Omega_z)$ is a classical solution of the following problem
\begin{equation*} 
\left \{ \begin{array} {ll} \displaystyle\medskip  w_t-d \Delta w=w(a-b w) \;\; &
\mbox{for $x \in\Omega(t), \; t>0$},\\
\displaystyle\medskip w=0, \; w_t=\mu|\nabla_x w|^2 \;&\mbox{for $x \in \partial\Omega(t), \; t>0$}, \\
w(0,x)=w_0 (x\cdot\nu_{z}-r_1- \xi_1\sin\theta) \;\; & \mbox{for $x \in \Omega_{z}(0)$},
\end{array} \right.
\end{equation*}
and hence by Theorems \ref{weakclassic} and \ref{unique}, it is the unique weak solution. 

By Lemma \ref{eslowra}, for any $\epsilon>0$, there exists $T_1=T_1(\epsilon)>0$ such that 
\begin{equation*}
\Omega_{z}(t)  \supset \Big\{x: \,x\cdot\nu_z\geq  -\Big(c_{*}-\frac{2}{3}\epsilon\Big)t +r_1+ \xi_1\sin\theta\Big\}\,\hbox{ for all } t\geq T_1,  
\end{equation*}
and
\begin{equation*}
\liminf_{t\to\infty}\left[{\inf}_{x\cdot \nu_z -r_1- \xi_1\sin\theta\geq -(c_{*}-\frac 23\epsilon) t} w_z(t,x)\right]\geq \frac{a}{b}.
\end{equation*}
Therefore, if we choose 
$$ T_2:=\max\Big\{T_1,\; \frac{3|r_1+ \xi_1\sin\theta|}{\epsilon}\Big\},$$ 
then
 \begin{equation*}
\Omega_{z}(t)  \supset \Big\{x: \,x\cdot\nu_z\geq  -(c_{*}-\epsilon)t \Big\}\,\hbox{ for all } t\geq T_2,  
\end{equation*}
and
\begin{equation*}
\liminf_{t\to\infty}\left[{\inf}_{x\cdot \nu_z \geq -(c_{*}-\epsilon) t} w_z(t,x)\right]\geq \frac{a}{b}.
\end{equation*}

On the other hand, by the choice of $w_0$ in \eqref{lowu0} and the property \eqref{initialcom}, we have 
$$\Omega_{z}(0)\subset  \Lambda_{z_1}  \subset \Omega(0), $$ 
and
$$w_0 (x\cdot\nu_z-r_1-\xi_1\sin\theta)\leq u_0(x) \,\hbox{ for } x\in \Omega_z(0).$$
Hence we can use Theorem \ref{thmcomp} to compare $u$ and $w_z$ and then obtain
$$w_z(t,x)\leq u(t,x) \,\hbox{ in }   [0,\infty)\times\R^N, $$ 
which clearly implies $\Omega_z(t) \subset \Omega(t)$ for $t\geq 0$.

Thus, we obtain
\begin{equation*}
 \Big\{x: \,x\cdot\nu_z\geq  -(c_{*}-\epsilon)t \Big\} \subset \Omega(t) \,\hbox{ for all } t\geq T_2,  
\end{equation*}
and 
\begin{equation*}
\liminf_{t\to\infty}\left[{\inf}_{x\cdot \nu_z \geq -(c_{*}-\epsilon) t} u(t,x)\right]\geq \frac{a}{b}.
\end{equation*}

Finally, by the arbitrariness of $z\in \partial \Lambda_{z_1}\backslash \{ z_1\}$, we obtain
\begin{equation*}
 \Lambda^\phi-\frac{c_*-\epsilon}{\sin\phi}t{e}_N\;= \bigcup_{z\in \partial \Lambda_{z_1}\backslash \{ z_1\} }\Big\{x: \,x\cdot\nu_z\geq  -(c_{*}-\epsilon)t \Big\}\; \subset\; \Omega(t) \,\hbox{ for all } t\geq T_2.  
\end{equation*}
The desired results then follow if we replace $\epsilon$ by $\tilde{\epsilon}: =\epsilon/ \sin\phi$.
\end{proof}

Next we prove 
\[
\Omega(t)\subset \Lambda^\phi-\left(\frac{c_*}{\sin\phi}+\epsilon\right)t\,{e}_N \mbox{ for all  large $t$}
\]
 by constructing a suitable weak supersolution to problem \eqref{eqfrfisher}. We do this with several lemmas.

\begin{lem}\label{lemupu0}
Let $u(t,x)$ and $\Omega(t)$ be given in the statement of Theorem \ref{spreadspeed}. Then for any $\delta>0$, there exist $t_2=t_2(\delta)>0$ and $r_2=r_2(\delta)>0$ such that
\begin{equation}\label{upu0}
u(t,x)\leq  \frac{a+\delta}{b-\delta} \,\hbox{ for }  t\geq t_2,\,x\in\R^N, 
\end{equation}
and 
\begin{equation}\label{conez3}
{\Omega(t_2)} \subset \Lambda_{z_2}:=\Lambda^\phi+\Big(\xi_2-\frac{r_2}{\sin\phi}\Big)e_N.
\end{equation}
\end{lem}

\begin{proof}
Let $u^*(t)$ be the unique solution of the problem 
$$\frac{d u^*}{dt}=u^*(a-bu^*) \hbox{ for }  t>0; \quad u^*(0)=\max\Big\{ \frac{a}{b},\, \|u_0\|_{L^{\infty}(\Omega_0)} \Big\}.  $$
Clearly, we have
$$u^*(t)\geq \frac{a}{b} \hbox{ for all } t\geq 0\,\hbox{ and }\,\lim_{t\to\infty} u^*(t)=\frac{a}{b}.$$
Moreover, it follows from Theorem \ref{asymu} and the parabolic comparison principle that  
$$u(t,x)\leq U(t,x)\leq u^*(t)\,\hbox{ for } t\geq 0,\, x\in\R^N,  $$
where $U(t,x)$ is the unique solution of the Cauchy problem \eqref{cauchyp}. As a consequence, for any $\delta>0$, there exists $t_2=t_2(\delta)>0$ such that 
$$u(t,x)\leq u^*(t) \leq \frac{a+\delta}{b-\delta}\,\hbox{ for all } t\geq t_2,\,x\in\R^N, $$
which clearly gives \eqref{upu0}. 

It remains to prove \eqref{conez3}. 
Let $t_2>0$ be determined as above. It then follows from Proposition \ref{pexsym} in the Appendix 
below that there exists $R_0>1$ depending  on $t_2$ such that, for any given
radially symmetric function $\hat{v}_0\in C^2([0,\infty))$  satisfying 
\begin{equation}\label{hatv0}
0<\hat{v}_0(r) < 2\| u_0 \|_{  L^{\infty}(\Omega_0)} \,\hbox{ for } \,r\in (0,\infty),\quad  \hat{v}_0(0)=0,\quad \|\hat{v}_0\|_{C^1([0,\infty)}\leq 3\| u_0 \|_{  L^{\infty}(\Omega_0)},
\end{equation}
the following free boundary problem 
\begin{equation}\label{sec4symubd}
 \left \{ \begin{array}{ll} 
\displaystyle\medskip \hat{v}_t-d \Delta \hat{v}=a\hat{v}-b\hat{v}^2, \;\; & 0<t<t_2, \; \hat{h}(t)<r<\infty,\\
\displaystyle\medskip \hat{v}(t, \hat{h}(t))=0, \;\; &0<t<t_2,\\
\displaystyle\medskip \hat{h}'(t)=-\mu \hat{v}_r (t, \hat{h}(t)), \;\; &0<t<t_2,\\
\hat{h}(0)=R_0, \;\;\; \hat{v}(0, r)=\hat{v}_0 (r-R_0), \;\; &R_0 \leq r<\infty, 
\end{array} \right.
\end{equation}
admits a unique classical solution $(\hat{v},\hat{h})$ defined for $0<t\leq t_2$ with $\hat{v}(t,r)>0$, $\hat{h}'(t)<0$ for  $0<t\leq t_2,\,\hat{h}(t) < r<\infty$, and
\begin{equation}\label{freedt2}
\hat{h}(t_2)\geq R_0/2,
\end{equation}
where due to the radial symmetry, $\Delta \hat{v}=\hat{v}_{rr}+\frac{N-1}{r}\hat{v}_r$.

Set
\[
r_2:=R_0+1 \quad\hbox{and}\quad  \Lambda_{z_2}:=\Lambda^\phi+\Big(\xi_2-\frac{r_2}{\sin\phi}\Big)e_N.
\]
Clearly 
\[
B_{r_2}(\hat x_0)\cap (\Lambda^\phi+\xi_2e_N)=\emptyset \;\hbox{ for all } \hat x_0\in \R^N\setminus \Lambda_{z_2}.
\]
Since $\Omega_0\subset (\Lambda^\phi+\xi_2e_N)$, it is easily seen that $|x-\hat{x}_0|-R_0>1$
for all $x\in\Omega_0$, $\hat x_0\in\R^N\setminus \Lambda_{z_2}$. In view of this, we may require that, in addition to the constraint \eqref{hatv0}, $\hat{v}_0$ also satisfies 
\begin{equation}\label{sec5inicom}
\hat{v}_0(|x-\hat{x}_0|-R_0) \geq u_0(x) \hbox{ for } x\in\Omega_0,\; \hat x_0\in\R^N\setminus \Lambda_{z_2}.  
\end{equation}
Indeed, this can be ensured by requiring $\hat{v}_0$ to satisfy $\hat{v}_0(r) \geq \| u_0 \|_{  L^{\infty}(\Omega_0)}$ for $r\geq 1$.

We now define, for each   $\hat x_0\in \R^N\setminus \Lambda_{z_2}$,
$$\hat{V}(t,x)=\hat{v}(t,|x-\hat{x}_0|)\,\hbox{ for }\,|x-\hat{x}_0|\geq \hat{h}(t),$$ 
and extend it to zero for $|x-\hat{x}_0|< \hat{h}(t)$ ($0\leq t\leq t_2$), then it is easily seen that
 $\hat{V}$ is the unique weak solution of the free boundary problem induced from \eqref{sec4symubd} over $[0,t_2]\times \R^N$ with initial function $\hat{v}_0(|x-\hat{x}_0|-R_0)$. 
Furthermore, due to \eqref{sec5inicom}, we conclude from Theorem \ref{thmcomp} that 
$$u(t,x)\leq \hat{V}(t,x) \,\hbox{ in }\, [0,t_2]\times \R^N.$$
This together with \eqref{freedt2} clearly implies  
\begin{equation*}
\Omega(t_2) \subset\Big\{x:\, |x-\hat{x}_0|\geq \hat{h}(t_2) \Big\} \subset \Big\{x:\, |x-\hat{x}_0|\geq R_0/2 \Big\} \;\hbox{ for all } \hat x_0\in\R^N\setminus \Lambda_{z_2}. 
\end{equation*}
It follows that
\[
\Omega(t_2)\subset \bigcap_{\hat x_0\in\R^N\setminus \Lambda_{z_2}}\Big\{x:\, |x-\hat{x}_0|\geq R_0/2 \Big\}\subset \Lambda_{z_2}.
\]
The proof of Lemma \ref{lemupu0} is now complete.
 \end{proof}

We are now ready to construct a weak supersolution of problem \eqref{eqfrfisher}. For any given small $\delta>0$, denote
$$
c^\delta:=c_*(\mu,a+\delta,b-\delta,d),
$$ 
and denote by $Z^{\delta}(r)$ the solution of \eqref{eqwave} with $k$, $a$, $b$ replaced by $c^{\delta}$, $a+\delta$, $b-\delta$, respectively. For $R>0$, we define
\[
\xi_R(t):=\xi_2-\frac{R+r_2}{\sin\phi}-\left[(1-\delta)^{-2}\frac{c^\delta}{\sin\phi}\right]\,t \;\hbox{ for } t\geq 0,
\]
with $r_2$ given in Lemma \ref{lemupu0} depending on $\delta>0$, and
\[
\Omega_R(t):=\Lambda_R^\phi+\xi_R(t)e_N \;\hbox{ for } t\geq 0,
\]
with 
$$ \Lambda_R^\phi:=\big\{x\in\Lambda^\phi: d(x,\partial\Lambda^\phi)>R\big\}.$$
Then define
\[
\overline u(t,x)= \overline u_R(t,x):=u_R(x-\xi_R(t)e_N) \;\hbox{ for } t\geq0,\;x\in\R^N,
\]
with
\[
u_R(x):=\left\{\begin{array}{ll}
\medskip
(1-\delta)^{-2}Z^\delta(d(x,\partial\Lambda_R^\phi)), & x\in\Lambda_R^\phi,\\
0, & x\not\in \Lambda_R^\phi.
\end{array}
\right.
\]
We are going to show that for suitably chosen $R$, $\overline u_R(t,x)$ is a weak supersolution to the equation satisfied by
$u(t_2+t,x)$; the desired result then easily follows.

\begin{lem}\label{weaksuper}
Let $\overline{u}$ be given as above. Then there exists $R=R(\delta)$ sufficiently large such that  $\overline u$
is a weak supersolution of \eqref{eqfrfisher} with $u_0(x)$ replaced by $\overline u(0,x)$.
\end{lem}

\begin{proof}
Let us observe that $\partial\Lambda^\phi_R$ is smooth and it can be decomposed into two parts, a spherical part
\[
\Sigma_R^1:=\partial B_R(0)\cap \Lambda^{\phi-\frac\pi 2},
\]
and part of the surface of the cone $\Lambda^\phi+\frac{R}{\sin\theta} e_N$ (recall $\theta=\pi-\phi$):
\[
\Sigma_R^2:=\Big(\partial\Lambda^\phi+\frac{R}{\sin\theta} e_N\Big)\setminus \Lambda^{\phi-\frac\pi 2}.
\]
Correspondingly, we can decompose $\Lambda_R^\phi$ into two parts:
$$\Lambda_R^\phi: =  \Lambda_{R,1}^\phi \cup \Lambda_{R,2}^\phi \;\hbox{ with }\;  \Lambda_{R,1}^\phi :=\Lambda_R^\phi \cap \Lambda^{\phi-\frac\pi 2}\; \mbox{ and }\; \Lambda_{R,2}^\phi:=\Lambda_R^\phi 
\setminus \Lambda^{\phi-\frac\pi 2}$$
($\Lambda_{R,1}^\phi$ and $\Lambda_{R,2}^\phi$ are illustrated in Figure \ref{figsupersolu}). 
\begin{figure}[h]
\centering
\def\svgwidth{9cm}
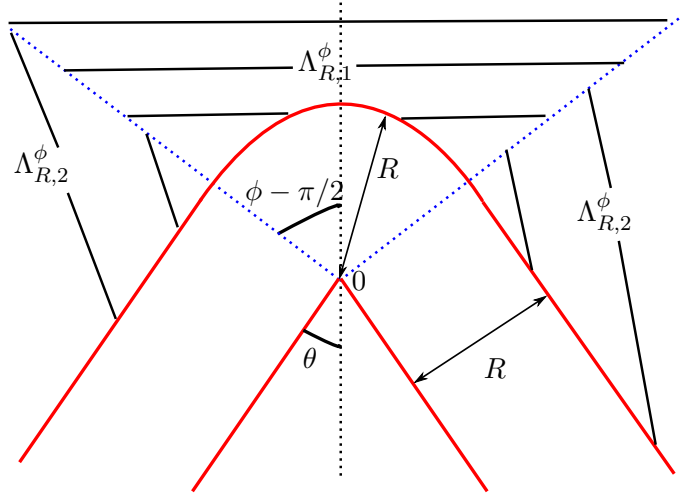
\caption{The domains $\Lambda_{R,1}^\phi$ and $\Lambda_{R,2}^\phi$ }\label{figsupersolu}
\end{figure}

In a similar way, for each $t\geq 0$, we can write
\[
\Omega_R(t) =\Omega_R^1(t)\cup \Omega_R^2(t),
\]
with 
$$
 \Omega_R^1(t):=\Lambda_{R,1}^\phi +\xi_R(t)  \; \mbox{ and }\; \Omega_R^2(t):=\Lambda_{R,2}^\phi +\xi_R(t).
$$

Clearly
\[
d(x,\partial\Lambda_R^\phi)=d(x, \Sigma_R^1)=|x|-R \mbox{ if } x\in \Lambda_{R,1}^\phi,
\]
and, by some simple geometrical calculations, for $x=(x', x_N)\in \Lambda_{R,2}^\phi$,
\[
d(x,\partial\Lambda_R^\phi)=d(x,\Sigma_R^2)=|x'|\cos\theta+x_N\sin\theta-R.
\]

It is straightforward to check that $\overline{u}$ and $\nabla_x \overline{u}$ are continuous in $\bigcup_{
t\geq 0} {\overline {{\Omega_R} (t)}}$,  that $\nabla_x^2 \overline{u}$, $\overline{u}_t$ are
continuous in $\bigcup_{t>0} \overline{\Omega_R (t)}$. 

Next, we show that for $R>0$ sufficiently large, 
\begin{equation}\label{checkmaineq}
\overline{u}_t-d\Delta\overline{u}\geq  a\overline{u}-b\overline{u}^2 \quad\hbox{for}\quad x\in{\Omega_R}(t),\,\,t>0.
\end{equation}
Denote $z:=x-\xi_R(t)e_N$; direct calculation shows that, for $x\in {\Omega}_R^1(t)$ and $t>0$, 
$$ 
\overline{u}_t= (1-\delta)^{-4} (Z^{\delta})'(z)\frac{z_N}{|z|}\frac{c^{\delta}}{\sin\phi},
$$
and 
$$d\Delta\overline{u}=(1-\delta)^{-2} \Big[ d(Z^{\delta})''(z)+\frac{d(N-1)}{|z|}(Z^{\delta})' (z) \Big].$$
Due to  $|z|>R$ and $z_N\geq |z|\sin\theta=|z|\sin\phi$ for $x\in {\Omega}_R^1(t)$ (i.e., $z\in \Lambda_{R,1}^\phi$), and $(Z^{\delta})'>0$, we have
\begin{equation*}
\left.\begin{array}{ll}
\displaystyle\medskip \overline{u}_t-d\Delta\overline{u}&\displaystyle \geq (1-\delta)^{-2}\Big[ (1-\delta)^{-2}c^{\delta}(Z^{\delta})'-\frac{d(N-1)}{|z|}(Z^{\delta})' - d(Z^{\delta})'' \Big]\\
\displaystyle\medskip&\displaystyle \geq (1-\delta)^{-2}\Big( \Big[(1-\delta)^{-2}c^{\delta}-\frac{d(N-1)}{R}\Big](Z^{\delta})'     - d(Z^{\delta})'' \Big)\\
\displaystyle\medskip &\displaystyle \geq (1-\delta)^{-2}\Big( c^{\delta}(Z^{\delta})' +\Big[\delta c^\delta-\frac{d(N-1)}{R}\Big](Z^{\delta})' - d(Z^{\delta})'' \Big).
\end{array}\right.
\end{equation*}
Therefore, if we choose 
\begin{equation}\label{chooser}
R\geq \frac{d(N-1)}{\delta c^\delta},
\end{equation}
 then for $x\in\Omega_R^1(t)$, 
\begin{equation*}
 \overline{u}_t-d\Delta\overline{u} \geq  (1-\delta)^{-2} \big[ c^{\delta}(Z^{\delta})' -  d(Z^{\delta})''\big]
 \geq (a+\delta)\overline{u}-(b-\delta)\overline{u}^2\geq a\overline{u}-b\overline{u}^2.
\end{equation*}

For $x\in {\Omega}_R^2(t)$ and $t>0$, it follows from a direct calculation that 
$$ \overline{u}_t= (1-\delta)^{-4}c^{\delta} (Z^{\delta})',$$
and 
$$d\Delta\overline{u}=(1-\delta)^{-2} \Big[ d(Z^{\delta})''+\frac{d(N-2)\cos\theta (Z^{\delta})' }{|x'|} \Big],$$
where $x':=(x_1,..., x_{N-1})\in \R^{N-1}$.

It is easily checked that for $z=(z', z_N)\in\Lambda_{R,2}^\phi$, we always have $|z'|\geq R\cos\theta$. Thus using $z=x-\xi_R(t)e_N\in\Lambda_{R,2}^\phi$,
we obtain
$$\frac{\cos\theta}{|x'|} \leq  \frac{1}{R}.$$
Thus for $R>0$ satisfying \eqref{chooser} and $x\in\Omega_R^2(t)$, we have  
\begin{equation*}
\left.\begin{array}{ll}
\displaystyle\medskip \overline{u}_t-d\Delta\overline{u}&\displaystyle \geq (1-\delta)^{-2}\Big[ (1-\delta)^{-2}c^{\delta}(Z^{\delta})'-\frac{d(N-2)}{R}(Z^{\delta})' - d(Z^{\delta})'' \Big]\\
\displaystyle\medskip &\displaystyle \geq (1-\delta)^{-2}\Big[ c^{\delta}(Z^{\delta})' +\Big(\delta c^\delta-\frac{d(N-2)}{R}\Big)(Z^{\delta})' - d(Z^{\delta})'' \Big]\\
 \displaystyle\medskip &\displaystyle \geq (1-\delta)^{-2}\Big[ c^{\delta}(Z^{\delta})'  - d(Z^{\delta})'' \Big]\\
&\geq  (a+\delta)\overline{u}-(b-\delta)\overline{u}^2,
\end{array}\right.
\end{equation*}
and hence, 
$$\overline{u}_t-d\Delta\overline{u} \geq   a\overline{u}-b\overline{u}^2.$$
We have thus proved that \eqref{checkmaineq} holds for all $R$ satisfying \eqref{chooser}. We henceforth fix such an $R$.

We next define
\[
\Phi(t,x)=R-d(x,\partial\Omega_R(t))=\left\{\begin{array}{ll}
\medskip R-|x-\xi_R(t)e_N|, & x\in\Omega_R^1(t),\smallskip\\
R-\big[|x'|\cos\theta+(x_N-\xi_R(t))\sin\theta\big], & x\in\Omega_R^2(t).
\end{array}
\right.
\]
Clearly, $\Phi$ is smooth, $\Omega_R(t)=\{x: \Phi(t,x)<0\}$ and $|\nabla_x\Phi|\not=0$ for $x\in\partial\Omega_R(t)$. We next show that
\begin{equation}\label{fbdy}
\Phi_t\leq \mu \nabla_x \overline{u}\cdot\nabla_x\Phi \quad \hbox{for} \quad  x\in \partial\Omega_R(t),\,\,t>0.
\end{equation}
 It is straightforward to calculate that, for $x\in 
\partial\Omega_R(t)$ and $t>0$, 
\begin{equation*}
\nabla_x \Phi \cdot \nabla_x \overline{u}= -(1-\delta)^{-2}(Z^{\delta})'(0),  
\end{equation*}
and
\begin{equation*}
\Phi_t(t,x) =\left\{ 
\begin{array}{ll}
\displaystyle\medskip -(1-\delta)^{-2}\frac{x_N-\xi_R(t)}{|x-\xi_R(t)e_N|}\frac{c^{\delta}}{\sin\theta}  &\displaystyle\medskip  \hbox{ if } x\in\Omega_R^1(t), \\
 -(1-\delta)^{-2}c^{\delta} &  \hbox{ if } x\in\Omega_R^2(t).
\end{array}\right.
\end{equation*}
On the other hand, it is easily seen that for any $z\in \Lambda_{R,1}^\phi$, 
$
{z_N}\geq |z| \sin\theta$.
It then follows that
\[
\frac{x_N-\xi_R(t)}{|x-\xi_R(t)e_N|}\geq \sin\theta \;\;\mbox{ for } x\in\Omega_R^1(t),
\]
and hence 
$$\Phi_t(t,x) \leq  -(1-\delta)^{-2}c^{\delta} \mbox{  for } x\in\Omega_R(t). $$
From this and $\mu(Z^{\delta})'(0)=c^{\delta}$, we deduce \eqref{fbdy}.
We may now apply Theorem \ref{super-sub} to conclude that $\overline u$ is a weak supersolution of \eqref{eqfrfisher} with $u_0(x)$ replaced by 
$\overline u(0,x)$.
\end{proof}
 
\begin{lem}\label{upbound}
Let $u(t,x)$ and $\Omega(t)$ be given in the statement of Theorem \ref{spreadspeed}. Then for any $\epsilon>0$, there exists $T_3=T_3(\epsilon)>0$ such that
$$ {\Omega(t)}  \subset \Lambda^\phi-\left(\frac{c_*}{\sin\phi}+\epsilon\right)t\,{e}_N   \,\hbox{ for all }  t\geq T_3.$$
\end{lem}
\begin{proof}
For any small $\delta>0$, let $t_2=t_2(\delta)$ and $R=R(\delta)$ be given in Lemma \ref{lemupu0} and Lemma \ref{weaksuper}, respectively. By Proposition \ref{semiwave}, there exists $r_3=r_3(\delta)>0$ such that 
\begin{equation}\label{waveupu0}
Z^{\delta}(r) \geq (1-\delta)\frac{a+\delta}{b-\delta}\; \hbox{ for all } r\geq r_3.  
\end{equation}

We first claim that
\begin{equation}\label{inidomcom}
\Omega(t_2) \subset   \Omega_R(0) 
\end{equation}
and
\begin{equation}\label{inicom}
u(t_2,x) \leq \overline{u}(0,x+\tilde r_3 e_N)\quad \hbox{for } \, x\in \Omega(t_2),
\end{equation}
where  
\[
\tilde r_3:=\frac{r_3}{\sin\theta}.
\] 
Indeed, it is easily seen from the definition that
\[
\Omega_R(0)\supset \Lambda_{z_2}=\Lambda^\phi+\Big(\xi_2-\frac{r_2}{\sin\phi}\Big)e_N.
\]
Thus,  \eqref{inidomcom} is a consequence of $\Omega(t_2)\subset \Lambda_{z_2}$ proved in Lemma \ref{lemupu0}.

 We now prove \eqref{inicom}. 
For any $x\in  \Omega(t_2)$, due to $\Omega(t_2)\subset \Lambda_{z_2}\subset \Omega_R(0)$, we obtain
\[\begin{array}{rl}
\displaystyle d\Big(x+\frac{r_3}{\sin\theta}e_N,\; \partial \Omega_R(0)\Big)\geq& \displaystyle d\Big(x+\frac{r_3}{\sin\theta}e_N,\; \partial \Lambda_{z_2}\Big)\medskip\\
=&\displaystyle d\Big(x,\;  \partial \Lambda_{z_2}-\frac{r_3}{\sin\theta}e_N\Big)\medskip\\
\geq &\displaystyle r_3+d(x,\;\partial\Lambda_{z_2})\geq r_3.
\end{array}
\]
Thus, for $x\in\Omega(t_2)$, due to $(Z^{\delta})'(r)>0$ in $(0,\infty)$,  we have 
\begin{equation*}
\left.\begin{array}{ll}
\displaystyle\medskip \overline{u}(0,x+\tilde r_3e_N)\!\!\! &\displaystyle\medskip= (1-\delta)^{-2}Z^{\delta}\Big(d\Big(x+\frac{r_3}{\sin\theta}e_N-\xi_{R}(0)e_N,\; \partial\Lambda_R^\phi\Big)\Big)\\
&\displaystyle\medskip= (1-\delta)^{-2}Z^{\delta}\Big(d\Big(x+\frac{r_3}{\sin\theta}e_N,\;\partial\Omega_R(0)\Big)\Big)\\
&\geq (1-\delta)^{-2}Z^\delta(r_3).
\end{array}\right.
\end{equation*} 
This together with \eqref{upu0} and \eqref{waveupu0}  implies that 
$$\overline{u}(0,x+\tilde r_3 e_N) \geq  (1-\delta)^{-1}\frac{a+\delta}{b-\delta}\geq u(t_2,x) \;\hbox{ for } x\in\Omega(t_2),$$
which proves \eqref{inicom}. 

By Lemma \ref{weaksuper},
$\overline{u}(t,x+\tilde r_3 e_N)$ is a weak supersolution of problem \eqref{eqfrfisher} with $u_0$ replaced by $\overline{u}(0,x+\tilde r_3 e_N)$, and since $u(t+t_2,x)$ is a weak solution of \eqref{eqfrfisher} with $u_0$ replaced by  $u(t_2,x)$, it follows from \eqref{inidomcom}, \eqref{inicom} and Theorem \ref{thmcomp} that 
$$u(t+t_2,x)\leq \overline{u}(t,x+\tilde r_3 e_N)\,\hbox{ in } [0,\infty)\times \R^N,$$
and hence,
$$
\Omega(t) \subset \Omega_R(t-t_2)-\tilde r_3 e_N=\Lambda^\phi_R+\left[ \xi_R(t-t_2)-\tilde r_3\right] e_N  \,\hbox{ for } \,t\geq t_2. 
$$
Since $\Lambda_R^\phi\subset \Lambda^\phi$,
we thus obtain
\[
\Omega(t)\subset \Lambda^\phi+\left[ \xi_R(t-t_2)-\tilde r_3\right] e_N  \,\hbox{ for } \,t\geq t_2. 
\]
Clearly,
\[
\xi_R(t-t_2)-\tilde r_3=M-(1-\delta)^{-2}\frac{c^\delta}{\sin\theta}\, t,
\]
with
\[M=M_\delta:=\xi_2-\frac{R+r_2+r_3}{\sin\theta}+(1-\delta)^{-2}\frac{c^\delta}{\sin\theta}\, t_2.
\]
 Since 
 \[
 \lim_{\delta\to 0} (1-\delta)^{-2}c^{\delta}=c_*,
 \]
  for any small $\epsilon>0$, we can find some $\delta'_\epsilon\in (0,\epsilon)$ such that 
\begin{equation*}
-(1-\delta'_{\epsilon})^{-2}\frac{c^{\delta'_{\epsilon}}}{\sin\theta}\geq - \frac{ c_* }{\sin\theta}-\frac\epsilon 2.   
\end{equation*}
We now fix $\delta=\delta'_{\epsilon}$ and obtain
\begin{equation*}
\xi_R(t-t_2)-\tilde r_3 \geq   M-\left(\frac{c_*}{\sin\theta}+\frac{\epsilon}{2}\right)t 
=   M+\frac{\epsilon}{2}t -\left(\frac{c_*}{\sin\theta}+{\epsilon}\right)t \geq -\left(\frac{c_*}{\sin\theta}+{\epsilon}\right)t
\end{equation*}
for $t\geq T_3$ with
\[
T_3=T_3(\epsilon):
= \max\Big\{t_2, \, \frac{2}{\epsilon} |M| \Big\}. 
\]
Therefore,
\[
\Omega(t)\subset \Lambda^\phi -\left(\frac{c_*}{\sin\theta}+{\epsilon}\right)t \,e_N \mbox{  for $t\geq T_3$,}
\]
as desired. 
\end{proof}

It is easily seen that \eqref{Omega(t)} in Theorem \ref{spreadspeed} follows from Lemmas \ref{lowbound} and \ref{upbound}, while \eqref{outconebehu} is a direct consequence of \eqref{infubound} and \eqref{upu0}. Thus,  Theorem \ref{spreadspeed} is now proved.

Theorem \ref{spreadspeed} implies that if $\Omega_0$ satisfies \eqref{outscone} with $\phi\in (\pi/2, \pi)$, then for all large time, the  free boundary $\partial\Omega(t)$ propagates to infinity in the negative $x_N$-direction with speed $c_*/\sin\phi$. Moreover, given any direction
$\nu\in\mathbb S^{N-1}$ pointing outward of $\Lambda^\phi$, if we denote 
\[
\psi:=\arccos \left[(-e_N)\cdot\nu\right] \mbox{ (and so $\psi\in (0,\theta)=(0,\pi-\phi)$)},
\]
then the spreading of $\Omega(t)$ in the direction $\nu$ is roughly at speed $c_*/(\sin\phi\cos\psi)$.
In sharp contrast, we will show in
 the following theorem,  that when $\Omega_0$ satisfies \eqref{outscone} with $\phi\in (0, \pi/2)$, the spreading of $\Omega(t)$ 
in a set of directions $\nu\in\mathbb S^{N-1}$ pointing outward of $\Lambda^\phi$, including $\nu=-e_N$, is roughly
at the speed $c_*$. (This set of directions $\nu$ is given by $\Sigma_\phi$ below.)

\begin{thm}\label{spreadspeed1}
Let $u(t,x)$ be the unique weak solution of problem \eqref{eqfrfisher} with $\Omega_0$ satisfying \eqref{outscone} for some $\phi
\in (0,\pi/2)$, and $u_0$ satisfying \eqref{assumeu0} and \eqref{addau0}.
Denote $\Omega(t)=\big\{x:\, u(t,x)>0\big\}$. Then for any $\epsilon>0$, there exists $\tilde{T}=\tilde{T}(\epsilon)>0$ such that
\begin{equation}\label{inconeset}
N\big[\Lambda^\phi, (c_*-\epsilon)t\big]\, \subset  \,{\Omega(t)}\, \subset \,  N\big[\Lambda^\phi, (c_*+\epsilon)t\big] \,\hbox{ for all }  t\geq \tilde{T}.
\end{equation}
Moreover, we have
\begin{equation}\label{inconebehu}
\lim_{t\to\infty}\left[{\sup}_{x \in N[\Lambda^\phi, (c_*-\epsilon)t]} \Big|u(t,x)-\frac{a}{b} \Big|\right]=0.
\end{equation}
\end{thm}

\begin{proof}
We prove this theorem by two steps.

\smallskip
\noindent
{\bf Step 1:} {\it Proof of the first relation of \eqref{inconeset} and \eqref{inconebehu}.} 

 Due to the assumptions \eqref{outscone} and \eqref{addau0}, we can find $\xi_0>\xi_1$ such that 
$$
u_0(x) \geq \sigma_0:= \frac{1}{2}\liminf_{d(\tilde{x},\partial \Omega_0)\to\infty,\; \tilde{x}\in\Omega_0} u_0(\tilde{x}) \hbox{ for all } 
x\in \Lambda^\phi+ \xi_0e_N. $$
 Let $R^*>0$ be the positive constant given  in \eqref{radius}. 
Then we choose a radial function $\tilde{v}_0\in C^2([0,R^*])$ such that 
$$0< \tilde{v}_0(x)\leq \sigma_0 \hbox{ in } [0,R^*),\quad \tilde{v}_0'(0)=\tilde{v}_0(R^*)=0.  $$
Next, for any fixed $x_0\in \R^N$ with $B_{R^*}(x_0) \subset \Lambda^\phi+\xi_0e_N$, set $r=|x-x_0|$, and consider the following radially symmetric free boundary problem 
\begin{equation}\label{eqsymball} 
\left \{ \begin{array}{ll} 
\displaystyle\medskip \tilde{v}_t-d \Delta \tilde{v}=\tilde{v} (a-b \tilde{v}), \;\;&t>0, \; 0<r<\tilde{k}(t), \\
\displaystyle\medskip \tilde{v}_r (t,0)=0, \;\; \tilde{v}(t,\tilde{k}(t))=0, \;\; &t>0,\\
\displaystyle\medskip \tilde{k}'(t)=-\mu v_r (t, \tilde{k}(t)), \;\; &t>0,\\
\displaystyle  \tilde{k}(0)=R^*, \;\;\; \tilde{v}(0,r)=\tilde{v}_0 (r), \;\; & 0 \leq r \leq R^*.
\end{array} \right.
\end{equation}
It follows from \cite[Theorems 2.1, 2.5 and Corollary 3.7]{DG1} that problem \eqref{eqsymball}  admits a unique classical solution $(\tilde{v}(t,r),\tilde{k}(t))$ defined for all $t\geq 0$, and 
\begin{equation}\label{speedradial}
\lim_{t\to\infty} \frac{\tilde{k}(t)}{t}= c_*.
\end{equation}
Furthermore, applying \cite[Theorem 6.4]{DG2} to this problem, we have 
\begin{equation}\label{asym-tilde-v}
\lim_{t\to\infty} \max_{r\leq (c_*-\epsilon/2)t }  \left| \tilde{v}(t,r)-\frac{a}{b} \right|=0
\end{equation}
for every small $\epsilon>0$.

On the other hand, by our choices of $\tilde{v}_0$ and $\xi_0$, we have 
$$\tilde{v}_0(|x-x_0|)\leq u_0(x) \quad\hbox{ for all }\, x\in  B_{R^*}(x_0). $$
Then extending $\tilde{v}(t,|x-x_0|)$ to be zero for $|x-x_0|> \tilde{k}(t)$ and applying the comparison principle Theorem \ref{thmcomp}, we obtain 
\begin{equation}\label{comp-tildev-u}
\tilde{v}(t,|x-x_0|) \leq u(t,x)  \quad\hbox{for all }\, t\geq 0, \,x\in\R^N, 
\end{equation}
and hence,
$$\Big\{x\in\R^N:\,\,|x-x_0|\leq \tilde{k}(t) \Big\}\subset \Omega(t) \quad\hbox{for all }\, t\geq 0. $$
This together with \eqref{speedradial} implies that, for any $\epsilon>0$, there exists $\tilde{T}_1=\tilde{T}_1(\epsilon)>0$ such that 
$$\Big\{x\in\R^N:\,\,|x-x_0|\leq \big(c_*-\frac\epsilon 2\big)t \Big\}\subset \Omega(t) \quad\hbox{for all }\, t\geq \tilde{T}_1.$$

Note that the above analysis remains valid if $x_0$ replaced by any point $\tilde{x}_0$ such that $B_{R^*}(\tilde{x}_0) \subset \Lambda^\phi+\xi_0e_N$, and that the constant $\tilde{T}_1$ is independent of the choice of such $\tilde{x}_0$. 
It then follows that 
$$\bigcup_{ B_{R^*}(x_0) \subset \Lambda^\phi+\xi_0e_N } \Big\{x\in\R^N:\,\,|x-x_0|\leq 
\big(c_*-\frac\epsilon 2\big)t \Big\}\,\subset\, \Omega(t) \quad\hbox{for all }\, t\geq \tilde{T}_1. $$ 
Furthermore, it is easily seen that there exists $\tilde{T}_2=\tilde{T}_2(\epsilon) \geq \tilde{T}_1$ such that,  for all  $ t\geq \tilde{T}_2$,
\begin{equation}\label{comp-N-II}
 N[\Lambda^\phi, (c_*-\epsilon)t]\,\subset\, \bigcup_{ B_{R^*}(x_0) \subset \Lambda^\phi+\xi_0e_N } \Big\{x\in\R^N:\,\,|x-x_0|\leq 
\big(c_*-\frac\epsilon 2\big)t \Big\}.
\end{equation}
 We thus obtain the first relation of \eqref{inconeset}.
 
To complete the proof of this step, it remains to show \eqref{inconebehu}.  On the one hand, by \eqref{comp-tildev-u} and \eqref{comp-N-II}, we have 
$${\inf}_{x \in N[\Lambda^\phi, (c_*-\epsilon)t]} u(t,x)\, \geq \,\inf \left\{ \min_{|x-x_0|\leq 
(c_*-\epsilon/2)t} \tilde{v}(t,|x-x_0|):\, {B_{R^*}(x_0) \subset \Lambda^\phi+\xi_0e_N} \right\}$$
for all $t\geq \tilde{T}_2$. It further follows from \eqref{asym-tilde-v} that 
$$\liminf_{t\to\infty}{\inf}_{x \in N[\Lambda^\phi, (c_*-\epsilon)t]} u(t,x) \,\geq \,\frac{a}{b}.$$
On the other hand, since $u_0$ is bounded, by the arguments used at the beginning of the proof of Lemma \ref{lemupu0}, we have
$$\limsup_{t\to\infty} \sup_{x\in\R^N} u(t,x) \leq \frac{a}{b}.$$
Combining the above, we immediately obtain \eqref{inconebehu}. 
 
\smallskip

\noindent
{\bf Step 2:} {\it Proof of the second relation of \eqref{inconeset}. }

 Choose a one-dimensional function $\tilde{w}_0\in C^2((-\infty,1])\cap L^{\infty}((-\infty,1])$ such that 
\begin{equation*}
\tilde{w}_0(y)\geq \| u_0 \|_{L^{\infty}(\Omega_0)} \hbox{ in } (-\infty,0],\quad \tilde{w}_0(x)>0 \hbox{ in } (0,1) \quad\hbox{and}\quad  \tilde{w}_0(1)=0.
\end{equation*}
Then we consider the following one-dimensional free boundary problem   
\begin{equation}\label{c2eqlow} 
\left \{ \begin{array}{ll} 
\displaystyle\medskip \tilde{w}_t-d \tilde{w}_{yy}=\tilde{w}(a-b \tilde{w}), \;\;&t>0, \; -\infty<y<\tilde{\rho}(t), \\
\displaystyle\medskip \tilde{w}(t,\tilde{\rho}(t))=0, \;\; &t>0,\\
\displaystyle\medskip \tilde{\rho}'(t)=-\mu \tilde{w}_y (t,\tilde{\rho}(t)), \;\; &t>0,\\
\displaystyle  \tilde{\rho}(0)=1, \;\;\; \tilde{w}(0,y)=\tilde{w}_0 (y), \;\; & -\infty<y\leq 1.
\end{array} \right.
\end{equation}
It follows from \cite[Theorem 2.11]{DDL} that, \eqref{c2eqlow} admits a (unique)
classical solution $(\tilde{w}(t,y), \tilde{\rho}(t))$ defined for all $t>0$ and
$\tilde{\rho}'(t)>0$, $\tilde{w}(t,y)>0$ for $ -\infty<y<\tilde{\rho}(t)$, $t>0$.

For any  
\[
\nu\in \Sigma_\phi:= \Big\{x\in\mathbb S^{N-1}: \arccos(x\cdot e_N)\in [\phi+\frac\pi 2, \pi]\Big\},
 \]
 define
$$\tilde{\Omega}_{\nu}(t)=\Big\{x\in\R^N: x\cdot \nu \leq \tilde{\rho}(t)\Big\}\quad\hbox{and}\quad
\tilde{w}_\nu(t,x)=\tilde{w}(t,x\cdot\nu).$$
Clearly, we have $\Omega_0-\xi_2e_N \subset \tilde{\Omega}_{\nu}(0)$, and $u_0(\cdot+\xi_2 e_N)\leq \tilde{w}_{\nu}(0,\cdot)$ in 
$\Omega_0-\xi_2e_N$.  Then by similar comparison arguments as those used in the proof of Lemma \ref{lowbound}, we obtain 
\begin{equation}\label{onedupsolu}
\Omega(t)-\xi_2e_N \subset  \tilde{\Omega}_{\nu}(t) \quad\hbox{for all }\,t\geq 0,\;\nu\in \Sigma_\phi.
\end{equation}

Furthermore, it follows from the proof of \cite[Theorem 4.2]{DL} with similar modifications as those given in the proof of Lemma \ref{eslowra} that, for the given $\epsilon>0$, there exists $\tilde{T}_3=\tilde{T}_3(\epsilon)>0$ such that 
\begin{equation*}
\tilde{\rho}(t) \leq  \big(c_{*}+\frac\epsilon 2\big) t\,\hbox{ for all } t\geq \tilde{T}_3.  
\end{equation*}
This together with \eqref{onedupsolu} implies 
$$
\Omega(t)-\xi_2 e_N \subset \Big\{x\in \R^N:\,\, x\cdot \nu \leq \big(c_{*}+\frac\epsilon 2\big) t \Big\} \, \hbox{ for all }\, t\geq \tilde{T}_3,\;\nu\in \Sigma_\phi.$$
Since
\[
N[\Lambda^\phi, \big(c_{*}+\frac\epsilon 2\big) t]=\bigcup_{\nu\in\Sigma_\phi}\Big\{x\in \R^N:\,\, x\cdot \nu \leq \big(c_{*}+\frac\epsilon 2\big) t \Big\},
\]
by enlarging $\tilde T_3$ if necessary (depending on $\xi_2$ and $\epsilon$), we obtain
\[
\Omega(t)\subset N\big[\Lambda^\phi, (c_*+\epsilon)t\big] \; \hbox{ for all }\, t\geq \tilde{T}_3.
\]
 which clearly gives the second relation of \eqref{inconeset}. 
\end{proof}

\begin{rem}\label{sec6rem2}
The estimates in Theorem \ref{spreadspeed1} can be improved  by making use of sharp estimates for the spreading speed  for 
one space dimension free boundary problems in \cite{DMZ1} and for radially symmetric free boundary problems 
in \cite{DMZ2}.   We leave the details to the interested reader. 
\end{rem}

\section{Appendix}
This appendix is concerned with the existence and uniqueness of classical solutions to an auxiliary radially symmetric problem with initial 
range the exterior of  a ball. These results have been used to construct the weak supersolution for problem \eqref{eqfrfisher} in the proof of Theorem \ref{spreadspeed}, and here we consider a more general problem which might have other applications.  

More precisely, for any given $T>0$, $C_1>0$ and $C_2>0$, we consider the following radially symmetric free boundary problem\begin{equation}\label{symubd}
 \left \{ \begin{array}{ll} 
\displaystyle\medskip v_t-d \Delta v=\tilde{g}(r,v), \;\; & 0<t<T, \; h(t)<r<\infty,\\
\displaystyle\medskip v(t, h(t))=0, \;\; &0<t<T,\\
\displaystyle\medskip h'(t)=-\mu v_r (t, h(t)), \;\; &0<t<T,\\
 h(0)=R_0, \;\;\; v(0, r)=v_0 (r-R_0), \;\; &R_0 \leq r<\infty, 
\end{array} \right.
\end{equation}
where $ \Delta v=v_{rr}+\frac{N-1}{r}v_r$, $R_0>1$ is a constant to be determined by $T,\,C_1,\,C_2$ later, and  $v_0$ is a given  function in $ C^2([0,\infty))$ satisfying
\begin{equation}\label{symini1}
0<v_0(r) \leq  C_1\,\hbox{ for } \,r\in (0,\infty),\quad  v_0(0)=0,\quad \|v_0\|_{C^1([0,\infty))}\leq C_2.  
\end{equation}
Here we assume that $\tilde{g}(r,v)$ is a continuous function defined over $\R^+\times \R^+$ satisfying 
\begin{equation}\label{tildeg} 
\left. \begin{array} {ll}
\hbox{(i)}& \mbox{$\tilde g(r,v)$ is H\"{o}lder continuous in $r\geq 0$, $v\geq 0$},\\
\hbox{(ii)} &\smallskip  \tilde{g}(r,v) \hbox{ is locally Lipschitz in } v \hbox{ uniformly for }r\geq 0,\,\,\\
\hbox{(iii)} &\smallskip  \tilde{g}(r,0) \equiv 0,\,\,\\
\hbox{(iv)} & \hbox{there exists } K>0 \hbox{ such that }  \tilde{g}(r,v)  \leq  Kv \hbox{ for all } r \geq 0 \hbox{ and } v \geq 0.
\end{array} \right\}
\end{equation}

\begin{prop}\label{pexsym}
Assume that \eqref{tildeg} is satisfied. For any $T>0$, $C_1>0$ and $C_2>0$, there exists a constant $R_0>1$ depending on $T$, $C_1$ and $C_2$ such that for any $v_0$ satisfying \eqref{symini1},  
problem \eqref{symubd} admits a unique solution $(v(t,r), h(t))$ with $h \in C^{1} ([0, T])$, $v \in C^{1,2}(D_T)$, where $D_T=\big\{(t,r): \; t \in [0,T], \; r \in [h(t), \infty)\big\}$. Moreover, 
$v(t,r)>0$, $h'(t)<0$ for  $0<t\leq T,\,h(t) < r<\infty$, and
\begin{equation}\label{simpty}
h(T)\geq \frac{R_0}{2}.
\end{equation}
\end{prop}

\begin{proof}  For given $R_0>1$, following the proof of \cite[Theorem 2.1]{DG1} 
we can show that \eqref{symubd} has a unique solution for some small $T>0$. The proof involves the straightening of the free boundary,
and different from \cite{DG1}, the resulting problem here is over an unbounded interval for the new space variable. However, the estimates easily carry over (by using suitable interior estimates, similar to a related situation treated in \cite{DL2}) and so we obtain the local existence and uniqueness all the same. Moreover, all the stated properties in the proposition, except \eqref{simpty}, also hold.

Furthermore, the solution can be extended as long as $h(t)>0$. Let $T_\infty=T_\infty(R_0)$ be the maximal existence time of the solution. If $h(t)\geq R_0/2$ for all $t\in (0, T_\infty)$, then necessarily $T_\infty=\infty$ and thus $h(T)> R_0/2$, and there is nothing left to prove.

Suppose now $h(t_0)<R_0/2$ for some $t_0\in (0, T_\infty)$. Since $h(0)=R_0$, we can find $T_0=T_0(R_0)\in (0, t_0)$ such that 
$h(T_0)=R_0/2$. We are going to show that $T_0(R_0)>T$ provided that $R_0$ is sufficiently large, which clearly implies
$h(T)>R_0/2$, as desired. We use an indirect argument and assume that $T_0(R_0)\leq T$ for all $R_0>1$. 

By the assumption \eqref{tildeg}, it follows from the parabolic comparison principle that $v(t,r)\leq \bar{v}(t)$ for $r>h(t)$ and $0\leq t\leq T_0$, where $\bar{v}(t)$ is the solution to 
$$\frac{d \bar{v}}{dt}=K\bar{v}  \hbox{ for } t>0; \quad \bar{v}(0)=\|v_0\|_{L^{\infty}([0,\infty))}.$$
Clearly, $\bar{v}(t)=\|{v}_0\|_{L^{\infty}([0,\infty))} \me^{Kt}$ for $t\geq 0$. Thus, we have
$$v(t,r) \leq C_3:=C_1 \me^{KT} \,\hbox{ for } \; r>h(t),\, 0\leq t\leq T_0.$$

 Next we prove that there exists a positive constant $C_4$ independent of $R_0$ such that
\begin{equation}\label{sec5esti1}
 -C_4 \leq h'(t)<0 \;\; \hbox{ for } \; t \in (0, T_0].
\end{equation}
This would lead to a contradiction, since it follows that
\[
R_0/2=h(T_0)\geq h(0)-C_4T_0\geq R_0-C_4T>R_0/2 \mbox{ for all  large } R_0.
\]
Therefore to complete the proof, it suffices to show \eqref{sec5esti1}.
 To this end, for $M>0$ to be determined later, we define
$$\Omega=\Omega_{M}:=\Big\{(t,r): \; 0<t\leq T_0, \;\; h(t)<r<h(t)+M^{-1} \Big\}$$
and construct an auxiliary function
$$w(t,r):=C_3 \big[2M (r-h(t))-M^2 (r-h(t))^2\big]\in (0, C_3)\;\mbox{ for } \; (t,r)\in \Omega.$$
We will show that for some suitable choice of $M>0$,  $w(t,r) \geq v(t,r)$ holds over $\Omega$.

Direct calculations give, for $(t,r) \in \Omega$,
$$w_t=2C_3 M \big(-h'(t)\big ) \big[1-M(r-h(t))\big] \geq 0,$$
and
 $$-w_r=-2C_3 M \big[1-M (r-h(t))\big]\geq -2C_3M,\quad -w_{rr}=2C_3 M^2. $$ 
 Making use of $1/2\leq R_0/2\leq h(t)$ for $t\in (0, T_0]$, we obtain, for $(t,r)\in\Omega$,
\begin{equation*}
 w_t-d \Big(w_{rr}+\frac{N-1}{r} w_r \Big) \geq 2dC_3 M^2-2dC_3M\frac{N-1}{r}\geq 2dC_3\big[M^2-2(N-1)M\big]. 
\end{equation*}
Thus, if we choose 
$$M\geq  (N-1)+\sqrt{\frac{K}{2d}+(N-1)^2},$$ 
where $K$ is given in \eqref{tildeg}, then 
\begin{equation*}
w_t-d \Big( w_{rr}+\frac{N-1}{r} w_r \Big) \geq KC_3\geq \tilde{g}(r,w) \mbox{ for } (t,r)\in\Omega. 
\end{equation*}
Let us also note that for $t\in (0,T_0]$,
$$w(t, h(t)+M^{-1})=C_3 \geq v(t, h(t)+M^{-1}), $$ and 
$$w(t,h(t))=0=v(t, h(t)).$$ 
Thus, if our choice of $M$ also ensures
$$v(0,r) \leq w(0, r) \,\hbox{ for }\, r \in [R_0, R_0+M^{-1}],$$
then we can apply the maximum principle to $w-v$ over $\Omega$ to
deduce that $v(t,r) \leq w(t,r)$ for $(t,r) \in  \Omega$. It would then follow that
$$v_r (t, h(t)) \leq w_r (t, h(t))=2MC_3,$$
and so
$$h'(t)=-\mu v_r (t,h(t)) \geq -C_4:=-2MC_3 \mu,$$
as we wanted.

To complete the proof, we calculate
$$w_r (0, r)=2C_3M \big[1-M(r-R_0)\big] \geq C_3 M \;\;\; \mbox{for $r \in [R_0, R_0+(2M)^{-1}]$}.$$ 
Therefore, upon choosing
$$M:=\max  \Big\{(N-1)+\sqrt{\frac{K}{2d}+(N-1)^2}, \;\; \frac{2 C_2}{C_3} \Big \},$$ 
we will have
$$w_r (0, r) \geq 2C_2> v_0'(r)=v_r(0,r) \;\;\; \mbox{for $[R_0, R_0+(2M)^{-1}]$}.$$ 
Since $w(0,R_0)=v(0,R_0)=0$, the above inequality implies
$$w(0,r) \geq v (0,r) \;\;\; \mbox{for $r \in [R_0, R_0+(2M)^{-1}]$}.$$ 

For $r \in [R_0+(2M)^{-1}, R_0+M^{-1}]$, we have $w(0,r) \geq C_3/2$ and
$$v(0,r) \leq \|v(0,\cdot) \|_{C^1 ([R_0,R_0+M^{-1}])} M^{-1}\leq C_2 M^{-1}
\leq \frac{C_3}{2},$$ 
which clearly gives $v (0,r) \leq w(0,r)$. Since $M$ is independent of $R_0$, this completes the proof of \eqref{sec5esti1}. 
\end{proof}


\end{document}